\documentclass{amsart}
\usepackage{amssymb}
\usepackage[mathcal]{eucal}
\usepackage{upref}

\usepackage{hyperref}

\usepackage[lite]{amsrefs}

\usepackage[all,cmtip]{xy}

\newtheoremstyle{slplain}% name
 {.5\baselineskip\@plus.2\baselineskip\@minus.2\baselineskip}% Space above
 {.5\baselineskip\@plus.2\baselineskip\@minus.2\baselineskip}% Space below
 {\slshape}% Body font
 {}%Indent amount (empty = no indent, \parindent = para indent)
 {\bfseries}%  Thm head font
 {.}%       Punctuation after thm head
 { }%      Space after thm head: " " = normal interword space;
 {}%       Thm head spec

\numberwithin{equation}{section}
\theoremstyle{slplain}
\newtheorem{thm}[equation]{Theorem}  % Numbered with the equation counter
\newtheorem{cor}[equation]{Corollary}     
\newtheorem{lem}[equation]{Lemma}         
\newtheorem{prop}[equation]{Proposition}  

\theoremstyle{definition}
\newtheorem{defn}[equation]{Definition} % Numbered with the equation counter

\theoremstyle{remark}
\newtheorem{ex}[equation]{Example}      % Numbered with the equation counter
\newtheorem{notation}[equation]{Notation}

\newcommand{\thmref}{Theorem~\ref}
\newcommand{\propref}{Proposition~\ref}
\newcommand{\corref}{Corollary~\ref}
\newcommand{\lemref}{Lemma~\ref}
\newcommand{\defref}{Definition~\ref}

\newcommand{\diagref}{Diagram~\ref}

\newcommand{\secref}{Section~\ref}

\newcommand*{\bs}{\boldsymbol}
\newcommand*{\bD}{\boldsymbol\Delta}

\newcommand{\cat}[1]{\mathcal{#1}}

\newcommand{\op}{^{\mathrm{op}}}

\newcommand{\LKan}{\mathbf{L}}

\newcommand{\diag}[1]{\bs{#1}}

\DeclareMathOperator{\diagon}{diag}

\DeclareMathOperator{\Tot}{Tot}

\newcommand{\real}[1]{\mathchoice
                       {\bigl|#1\bigr|}%
                       {\bigl|#1\bigr|}%
                       {|#1|}%
                       {|#1|}}

\newcommand{\cs}{\mathrm{cs}_{*}}

\newcommand{\match}{\mathrm{M}}

\newcommand{\matchcat}[2]{\mathchoice
  {\bdry\undercat{\inv{#1}}{#2}}
  {\bdry\undercat{\inv{#1}}{#2}}
  {\bdry\undercat{\inv{#1}}{#2}}
  {\bdry\undercat{\inv{#1}}{#2}}}

\newcommand{\inv}[1]{\overleftarrow{#1}}

\newcommand{\undercat}[2]{%
  {(\mathord{#2}\nonscript\,\mathord\downarrow\nonscript\,\mathord{#1})}}

\newcommand{\overcat}[2]{%
  {(\mathord{#1}\nonscript\,\mathord\downarrow\nonscript\,\mathord{#2})}}

\newcommand{\bovercat}[2]{%
  {\bigl(\mathord{#1}\nonscript\,\mathord\downarrow
             \nonscript\,\mathord{#2}\bigr)}}

\newcommand{\Top}{\mathrm{Top}}
\newcommand{\SiS}{\mathrm{SS}}

\newcommand{\B}{\mathrm{B}}

\newcommand{\suchthat}{\mid}

\DeclareMathOperator{\Ob}{Ob}

\newcommand{\bdry}{\partial}

\DeclareMathOperator*{\holim}{holim}
\DeclareMathOperator{\Map}{Map}
\DeclareMathOperator{\map}{map}

\newcommand{\we}{\cong}
\newcommand{\iso}{\approx}

\newcommand{\pullback}[3]{#1\mathbin{\mathord{\times}_{\!#2}}#3}

\newcommand{\Period}{\rlap{\enspace .}}

\newcommand{\diXhat}{\smash{\widehat{\diag X}}\vphantom{\widehat X}}

\newcommand{\obDn}{[\vec{p}]}
\newcommand{\obUC}{[k]^{n} \to \obDn}

\begin{document}

\title{The diagonal of a multicosimplicial object}

\author{Philip S. Hirschhorn}

\address{Department of Mathematics\\
   Wellesley College\\
   Wellesley, Massachusetts 02481}

\email{psh@math.mit.edu}

\urladdr{http://www-math.mit.edu/\textasciitilde psh}

\date{August 23, 2015}


\maketitle

\tableofcontents

%--------------------------------------------------------------------
%--------------------------------------------------------------------
\section{Introduction}
\label{sec:Intro}

After constructing a multicosimplicial object, it is common to pass to
its diagonal cosimplicial object.  In order for the total object of
the diagonal to have homotopy meaning, though, you need to know that
that diagonal is fibrant.

We prove in \thmref{thm:RiQuil} that the functor that takes a
multicosimplicial object in a model category to its diagonal
cosimplicial object is a right Quillen functor.  This implies that the
diagonal of a fibrant multicosimplicial object is a fibrant
cosimplicial object, which has applications to the calculus of
functors (see \cite{B-E-J-M}).  We also show in \thmref{thm:NotQE} and
\corref{cor:notqeTop} that, although the diagonal functor is a Quillen
functor, it is not a Quillen equivalence for multicosimplicial spaces.

In \secref{sec:TotOb} we discuss total objects, and show that the
total object of a multicosimplicial object is isomorphic to the total
object of the diagonal.

In \secref{sec:holim} we discuss homotopy limits, and show that the
diagonal embedding of the cosimplicial indexing category into the
multicosimplicial indexing category is homotopy left cofinal, which
implies that the homotopy limits are weakly equivalent if the
multicosimplicial object is at least objectwise fibrant.

%--------------------------------------------------------------------
%--------------------------------------------------------------------
\section{Definitions and notation}
\label{sec:Not}

\begin{notation}
  \label{not:cosimp}
  If $n$ is a nonnegative integer, we let $[n]$ denote the ordered set
  $(0, 1, 2,\ldots, n)$.  We will use $\bD$ to denote the
  \emph{cosimplicial indexing category}, which is the category with
  objects the $[n]$ for $n \ge 0$ and with morphisms $\bD\bigl([n],
  [k]\bigr)$ the weakly monotone functions $[n] \to [k]$.
\end{notation}

\begin{defn}
  \label{def:cosimp}
  If $\cat M$ is a category, a \emph{cosimplicial object in $\cat M$}
  is a functor $\bD \to \cat M$, and the \emph{category of
    cosimplicial objects in $\cat M$} is the functor category $\cat
  M^{\bD}$.  If $\diag X$ is a cosimplicial object in $\cat M$, then
  we will generally denote the value of $\diag X$ on $[k]$ as $\diag
  X^{k}$.
\end{defn}

\begin{notation}
  \label{not:multi}
  If $n$ is a positive integer, then we will let $\bD^{n}$ denote the
  product category $\underbrace{\bD \times \bD \times \cdots \times
    \bD}_{\text{$n$ times}}$.
\end{notation}

\begin{defn}
  \label{def:multicosimp}
  If $\cat M$ is a category and $n$ is a positive integer, then an
  \emph{$n$-cosimplicial object in $\cat M$} is a functor $\bD^{n} \to
  \cat M$.  If $\diag X$ is an $n$-cosimplicial object in $\cat M$,
  then we will generally denote the value of $\diag X$ on
  $\bigl([k_{1}], [k_{2}], \ldots, [k_{n}]\bigr)$ by $\diag X^{(k_{1},
    k_{2}, \ldots, k_{n})}$.
\end{defn}

%--------------------------------------------------------------------
\subsection{The diagonal}
\label{sec:diag}

An $n$-cosimplicial object in a category $\cat M$ is a functor $\diag
X\colon \bD^{n} \to \cat M$, and we can restrict that functor to the
``diagonal subcategory'' of $\bD^{n}$ to obtain a cosimplicial object
$\diagon\diag X$ in $\cat M$.

\begin{defn}
  \label{def:diag}
  Let $n$ be a positive integer.
  \begin{enumerate}
  \item The \emph{diagonal embedding} of the category $\bD$ into
    $\bD^{n}$ is the functor $D\colon \bD \to \bD^{n}$ that takes the
    object $[k]$ of $\bD$ to the object $\bigl(\underbrace{[k], [k],
      \ldots, [k]}_{\text{$n$ times}}\bigr)$ of $\bD^{n}$ and the
    morphism $\phi\colon [p] \to [q]$ of $\bD$ to the morphism
    $(\phi^{n})$ of $\bD^{n}$.
  \item If $\cat M$ is a category and $\diag X$ is an $n$-cosimplicial
    object in $\cat M$, then the \emph{diagonal} $\diagon\diag X$ of
    $\diag X$ is the cosimplicial object in $\cat M$ that is the
    composition $\bD \xrightarrow{D} \bD^{n} \xrightarrow{\diag X}
    \cat M$, so that $(\diagon\diag X)^{k} = \diag X^{(k, k, \ldots,
      k)}$.
\end{enumerate}

\end{defn}

%--------------------------------------------------------------------
\subsection{Matching objects}
\label{sec:match}

If $\cat C$ is a Reedy category (see \cite{MCATL}*{Def.~15.1.2}),
$\cat M$ is a model category, $\diag X$ is a $\cat C$-diagram in $\cat
M$, and $\alpha$ is an object of $\cat C$, then we will use the
notation of \cite{MCATL}*{Def.~15.2.5} and denote the \emph{matching
  object} of $\diag X$ at $\alpha$ by $\match_{\alpha}\diag X$, or by
$\match_{\alpha}^{\cat C}\diag X$ if the indexing category isn't
obvious.

Note that in the case of cosimplicial objects, our notation for
matching objects (see \cite{MCATL}*{Def.~15.2.5}) differs from that of
\cite{YM}*{Ch.~X, \S 4}, in that we index a matching object by the
degree at which it is the matching object, whereas \cite{YM}*{Ch.~X,
  \S 4} indexes it by one less than that.  Thus, our notation for the
matching map of a cosimplicial object $\diag X$ at $[k]$ is $\diag
X^{k} \to \match_{k}\diag X$, while the notation of \cite{YM}*{Ch.~X,
  \S 4} is $\diag X^{k} \to \match^{k-1}\diag X$.

\begin{defn}
  \label{def:Matchobj}
  Let $\cat C$ be a Reedy category, let $\cat M$ be a model category,
  let $\diag X$ be a $\cat C$-diagram in $\cat M$, and let $\alpha$ be
  an object of $\cat C$.
  \begin{enumerate}
  \item The \emph{matching category} $\matchcat{\cat C}{\alpha}$ of
    $\cat C$ at $\alpha$ is the full subcategory of
    $\undercat{\inv{\cat C}}{\alpha}$ containing all of the objects
    except the identity map of $\alpha$.
  \item The \emph{matching object} of $\diag X$ at $\alpha$ is
    $\match_{\alpha}\diag X = \lim_{\matchcat{\cat C}{\alpha}} \diag
    X$ and the \emph{matching map} of $\diag X$ at $\alpha$ is the
    natural map $\diag X_{\alpha} \to \match_{\alpha}\diag X$.  We
    will use $\match_{\alpha}^{\cat C}\diag X$ to denote the matching
    object if the indexing category isn't obvious.
  \end{enumerate}
\end{defn}

\begin{defn}[\cite{MCATL}*{Def.~15.3.3}]
  \label{def:RFib}
  Let $\cat C$ be a Reedy category and let $\cat M$ be a model
  category.  A map $f\colon \diag X \to \diag Y$ of $\cat C$-diagrams
  in $\cat M$ is a \emph{Reedy fibration} if for every object $\alpha$
  of $\cat C$ the relative matching map $\diag X_{\alpha} \to
  \pullback{\diag Y_{\alpha}}{\match_{\alpha} \diag Y}{\match_{\alpha}
    \diag X}$ is a fibration in $\cat M$.
\end{defn}

%--------------------------------------------------------------------
%--------------------------------------------------------------------
\section{The diagonal is a right Quillen functor}
\label{sec:MnThm}

We prove in \thmref{thm:RiQuil} that the functor that takes a
multicosimplicial object to its diagonal cosimplicial object is a
right Quillen functor.

\begin{defn}
  \label{def:Pdelta}
  Let $\cat M$ be a model category, let $n$ be a positive integer, and
  let $\diag X \to \diag Y$ be a Reedy fibration of $n$-cosimplicial
  objects in $\cat M$.  For every nonnegative integer $k$, the
  matching objects of $\diag X$ in $\cat M^{\bD^{n}}$ at
  $\bigl([k],[k],\ldots, [k]\bigr)$ and of $\diagon\diag X$ in $\cat
  M^{\bD}$ at $[k]$ are
  \begin{displaymath}
    \match_{([k],[k],\ldots,[k])}^{\bD^{n}} \diag X =
    \lim_{\matchcat{\bD^{n}}{([k],[k],\ldots,[k])}} \diag X
    \qquad\text{and}\qquad
    \match_{[k]}^{\bD} \diagon\diag X =
    \lim_{\matchcat{\bD}{[k]}} \diagon\diag X
  \end{displaymath}
  (with similar formulas for $\diag Y$), and we define
  $P_{k}^{\bD^{n}}$ and $P_{k}^{\bD}$ by letting the following
  diagrams be pullbacks:
  \begin{displaymath}
    \vcenter{
      \xymatrix{
        {P_{k}^{\bD^{n}}} \ar@{..>}[r] \ar@{..>}[d]
        & {\diag Y^{(k,k,\ldots,k)}} \ar[d]\\
        {\match_{([k],[k],\ldots,[k])}^{\bD^{n}} \diag X}
        \ar[r]
        & {\match_{([k],[k],\ldots,[k])}^{\bD^{n}} \diag Y}
      }% xymatrix
    }% vcenter
    \qquad
    \vcenter{
      \xymatrix{
        {P_{k}^{\bD}} \ar@{..>}[r] \ar@{..>}[d]
        & {\diag Y^{(k,k,\ldots,k)}} \ar[d]\\
        {\match_{[k]}^{\bD} \diagon\diag X}
        \ar[r]
        & {\match_{[k]}^{\bD} \diagon\diag Y \Period}
      }% xymatrix
    }% vcenter
  \end{displaymath}
  Thus,
  \begin{itemize}
  \item if the map $\diag X \to \diag Y$ is a Reedy fibration of
    $n$-cosimplicial objects then the natural map $\diag
    X^{(k,k,\ldots,k)} \to P_{k}^{\bD^{n}}$ is a fibration for all $k
    \ge 0$, and
  \item the map $\diagon\diag X \to \diagon\diag Y$ is a Reedy
    fibration of cosimplicial objects if and only if the natural map
    $\diag X^{(k,k,\ldots,k)} \to P_{k}^{\bD}$ is a fibration for all
    $k \ge 0$
  \end{itemize}
  (see \defref{def:RFib}).
\end{defn}

\begin{prop}
  \label{prop:matchmatch}
  Let $\cat M$ be a model category, let $n$ be a positive integer, and
  let $\diag X \to \diag Y$ be a Reedy fibration of $n$-cosimplicial
  objects in $\cat M$.  Since we are viewing $\bD$ as a subcategory of
  $\bD^{n}$, for every nonnegative integer $k$ there are natural maps
  \begin{align*}
    \lim_{\matchcat{\bD^{n}}{([k],[k],\ldots,[k])}} \diag X
    &\longrightarrow \lim_{\matchcat{\bD}{[k]}} \diagon\diag X
    \qquad\text{and}\\
    \lim_{\matchcat{\bD^{n}}{([k],[k],\ldots,[k])}} \diag Y
    &\longrightarrow \lim_{\matchcat{\bD}{[k]}} \diagon\diag Y
  \end{align*}
  and those induce a natural map
  \begin{displaymath}
    P_{k}^{\bD^{n}} \longrightarrow P_{k}^{\bD}
  \end{displaymath}
  (see \defref{def:Pdelta}).  That natural map is a fibration.
\end{prop}

The proof of \propref{prop:matchmatch} is in \secref{sec:Proof}.

\begin{thm}
  \label{thm:fibrant}
  If $\cat M$ is a model category, $n$ is a positive integer, and
  $\diag X \to \diag Y$ is a Reedy fibration of $n$-cosimplicial
  objects in $\cat M$, then the induced map of diagonals
  \begin{displaymath}
    \diagon\diag X \longrightarrow \diagon\diag Y
  \end{displaymath}
  is a Reedy fibration of cosimplicial objects.
\end{thm}

\begin{proof}
  We must show that for every nonnegative integer $k$ the map
  \begin{displaymath}
    (\diagon\diag X)^{k} = \diag X^{(k,k,\ldots,k)} \longrightarrow
    P_{k}^{\bD}
  \end{displaymath}
  (see \defref{def:Pdelta}) is a fibration.  That map is the
  composition
  \begin{displaymath}
    \diag X^{(k,k,\ldots,k)} \longrightarrow
    P_{k}^{\bD^{n}} \longrightarrow
    P_{k}^{\bD} \Period
  \end{displaymath}
  The first of those is a fibration because the map $\diag X \to \diag
  Y$ is a Reedy fibration of $n$-cosimplicial objects, and
  \propref{prop:matchmatch} is the statement that the second is also a
  fibration.
\end{proof}

Special cases of the following corollary are already known (see
\cite{shipley}*{Lem.~7.1}, in view of
\cite{dwy-mil-neis}*{Prop.~5.8}).
\begin{cor}
  \label{cor:fibrant}
  If $\cat M$ is a model category, $n$ is a positive integer, and
  $\diag X$ is a Reedy fibrant $n$-cosimplicial object in $\cat M$,
  then the diagonal cosimplicial object $\diagon\diag X$ is Reedy
  fibrant.
\end{cor}

\begin{proof}
  This follows from \thmref{thm:fibrant} by letting $\diag Y$ be the
  constant $n$-cosimplicial object at the terminal object of $\cat M$.
\end{proof}

\begin{thm}
  \label{thm:RiQuil}
  If $\cat M$ is a model category and $n$ is a positive integer, then
  the diagonal functor $\diagon\colon \cat M^{\bD^{n}} \to \cat
  M^{\bD}$, which takes an $n$-cosimplicial object in $\cat M$ to its
  diagonal cosimplicial object, is a right Quillen functor (see
  \cite{MCATL}*{Def.~8.5.2}).
\end{thm}

\begin{proof}
  Since a model category is cocomplete, the left Kan extension of a
  cosimplicial object along the diagonal inclusion $\bD \to \bD^{n}$
  always exists (see \cite{borceux-I}*{Thm.~3.7.2} or
  \cite{schubert}*{Thm.~17.1.6} or, for the dual statement,
  \cite{McL:categories}*{Cor.~X.3.2}), and so the diagonal functor has
  a left adjoint.  Thus, we need only show that the diagonal functor
  preserves both fibrations and trivial fibrations (see
  \cite{MCATL}*{Prop.~8.5.3}).

  \thmref{thm:fibrant} implies that that the diagonal preserves
  fibrations.  Since the weak equivalences of cosimplicial objects and
  of $n$-cosimplicial objects are defined degreewise, the restriction
  of an $n$-cosimplicial object to its diagonal preserves \emph{all}
  weak equivalences, and so it also preserves trivial fibrations.
\end{proof}

%--------------------------------------------------------------------
\subsection{The multicosimplicial product of standard simplices}
\label{sec:CosSS}

The main result of this section is \thmref{thm:LeftKan}, which we will
use in the following section to show that the right Quillen functor of
\thmref{thm:RiQuil} is not a Quillen equivalence.  It will also be
used in \thmref{thm:IsoTot} to show that the total object of a
multicosimplicial object is isomorphic to the total object of its
diagonal cosimplicial object.

\begin{defn}
  \label{def:MultSS}
  If $F\colon \cat A \to \cat B$ is a functor between small categories
  and $X$ is an object of $\cat B$, then $\cat B(F-,X)$ is the $\cat
  A\op$-diagram of sets that on an object $W$ of $\cat A$ is the set
  $\cat B(FW, X)$.  This is natural in $X$, and thus defines a $\cat
  B$-diagram of $\cat A\op$-diagrams of sets.

  If $n$ is a positive integer and $F\colon \cat A \to \cat B$ is the
  diagonal embedding $D\colon \bD \to \bD^{n}$ (see
  \defref{def:diag}), then an $\cat A\op$-diagram of sets is a
  $\bD\op$-diagram of sets, i.e., a simplicial set, and so this
  defines a $\bD^{n}$-diagram of simplicial sets, i.e., an
  $n$-cosimplicial simplicial set, which we will denote by
  $\Delta^{(n)}$.  If $\bigl([p_{1}], [p_{2}], \ldots, [p_{n}]\bigr)$
  is an object of $\bD^{n}$, then the simplicial set $\bD^{n}\bigl(D-,
  ([p_{1}], [p_{2}], \ldots, [p_{n}])\bigr)$ has as its $k$-simplices
  the $n$-tuples of maps
  \begin{displaymath}
    (\alpha_{1}, \alpha_{2}, \ldots, \alpha_{n})\colon
    \bigl([k], [k], \ldots, [k]\bigr) \longrightarrow
    \bigl([p_{1}], [p_{2}], \ldots, [p_{n}]\bigr)
  \end{displaymath}
  where each $\alpha_{i}\colon [k] \to [p_{i}]$ is a weakly monotone
  map.  Thus, each $k$-simplex is the product for $1 \le i \le n$ of a
  $k$-simplex of $\Delta[p_{i}]$, i.e., a $k$-simplex of
  $\Delta[p_{1}] \times \Delta[p_{2}] \times \cdots \times
  \Delta[p_{n}]$, and so $\Delta^{(n)}\bigl([p_{1}], [p_{2}], \ldots,
  [p_{n}]\bigr)$ is the product of standard simplices $\Delta[p_{1}]
  \times \Delta[p_{2}] \times \cdots \times \Delta[p_{n}]$.  That is,
  $\Delta^{(n)}\colon \bD^{n} \to \SiS$ is an $n$-cosimplicial
  simplicial set whose value on the object $\bigl([p_{1}], [p_{2}],
  \ldots, [p_{n}]\bigr)$ is $\Delta[p_{1}] \times \Delta[p_{2}] \times
  \cdots \times \Delta[p_{n}]$.  We will call it the
  \emph{$n$-cosimplicial product of standard simplices}.

  If $n = 1$, so that $F\colon \cat A \to \cat B$ is the identity
  functor of $\bD$, then this defines a cosimplicial object in the
  category of simplicial sets, i.e., a cosimplicial simplicial set,
  which we will denote by $\Delta$.  If $k$ is a nonnegative integer,
  then for a nonnegative integer $i$ the $i$-simplices of $\bD(-,
  [k])$ are the maps $\bD([i],[k])$, i.e., the weakly monotone
  functions $[i] \to [k]$, and so the simplicial set $\bD(-,[k])$ is
  the \emph{standard $k$-simplex}, which we will denote by
  $\Delta[k]$.  That is, $\Delta\colon \bD \to \SiS$ is a cosimplicial
  simplicial set, and its value on the object $[k]$ is $\Delta[k]$.
  We will call it the \emph{cosimplicial standard simplex}.
\end{defn}

\begin{lem}
  \label{lem:prodsimpcof}
  If $n$ is a positive integer, then the $n$-cosimplicial product of
  standard simplices $\Delta^{(n)}$ (see \defref{def:MultSS}) is a
  Reedy cofibrant (see \cite{MCATL}*{Def.~15.3.3}) $n$-cosimplicial
  simplicial set.
\end{lem}

\begin{proof}
  The latching map at the object $\bigl([p_{1}], [p_{2}], \ldots
  [p_{n}]\bigr)$ of $\bD^{n}$ is the inclusion of the boundary of
  $\Delta[p_{1}]\times \Delta[p_{2}] \times \cdots \times
  \Delta[p_{n}]$, and is thus a cofibration.
\end{proof}

\begin{lem}
  \label{lem:ColimCtSmp}
  If $K$ is a simplicial set and $\Delta K$ is the category of
  simplices of $K$ (which has as objects the simplices of $K$ and as
  morphisms from $\sigma$ to $\tau$ the simplicial operators that take
  $\tau$ to $\sigma$; see \cite{MCATL}*{Def.~15.1.16}), then $K$ is
  naturally isomorphic to the colimit of the $\Delta K$-diagram of
  simplicial sets that
  \begin{itemize}
  \item takes a $k$-simplex of $K$ to the standard $k$-simplex
    $\Delta[k]$,
  \item when $\bdry_{i}(\tau) = \sigma$, takes $\bdry_{i}$ to the
    inclusion of the image of $\sigma$ as the $i$'th face of the image
    of $\tau$, and
  \item when $s^{i}(\tau) = \sigma$, takes $s^{i}$ to the collapse of
    the image of $\sigma$ to the image of $\tau$ that identifies
    vertices $i$ and $i+1$
  \end{itemize}
  under an isomorphism that for a $k$-simplex $\sigma$ takes the
  nondegenerate $k$-simplex of $\Delta[k]$ to $\sigma$.
\end{lem}

\begin{proof}
  See \cite{MCATL}*{Prop.~15.1.20}.
\end{proof}

The following theorem is the degree $0$ part of
\cite{dwy-mil-neis}*{Remark on p.~172} and
\cite{shipley}*{Prop.~8.1}.  The full statement, for a general model
category, is \thmref{thm:IsoTot}.

\begin{thm}
  \label{thm:LeftKan}
  If $n$ is a positive integer, then the left Kan extension of the
  cosimplicial standard simplex $\Delta$ (see \defref{def:MultSS})
  along the diagonal embedding $\bD \to \bD^{n}$ (see
  \defref{def:diag}) is the $n$-cosimplicial product of standard
  simplices $\Delta^{(n)}$ (see \defref{def:MultSS}) with the natural
  transformation $\alpha\colon \Delta \to \diagon\Delta^{(n)}$ that on
  the object $[k]$ of $\bD$ is the diagonal map $\Delta[k] \to
  \Delta[k]\times \Delta[k]\times \cdots \times \Delta[k]$.  Thus, for
  every $n$-cosimplicial simplicial set $\diag X$ there is a natural
  isomorphism of sets
  \begin{displaymath}
    \SiS^{\bD^{n}}(\Delta^{(n)}, \diag X) \iso
    \SiS^{\bD}(\Delta, \diagon\diag X) 
  \end{displaymath}
  between the set of maps of $n$-cosimplicial simplicial sets
  $\Delta^{(n)} \to \diag X$ and the set of maps of cosimplicial
  simplicial sets $\Delta \to \diagon\diag X$ (see \defref{def:diag}).
  That isomorphism takes a map $f\colon \Delta^{(n)} \to \diag X$ in
  $\SiS^{\bD^{n}}$ to the composition $\Delta \xrightarrow{\alpha}
  \diagon\Delta^{(n)} \xrightarrow{\diagon f} \diagon\diag X$ in
  $\SiS^{\bD}$ (see \cite{borceux-I}*{Def.~3.7.1}).
\end{thm}

\begin{proof}
  Since the category of simplicial sets is cocomplete, the left Kan
  extension $\LKan\Delta$ of $\Delta$ exists and can be constructed
  pointwise (see \cite{borceux-I}*{Thm.~3.7.2}).  We view $\bD$ as the
  diagonal subcategory of $\bD^{n}$, and so for each object
  $\bigl([p_{1}], [p_{2}], \ldots, [p_{n}]\bigr)$ of $\bD^{n}$, the
  simplicial set $\LKan\Delta\bigl([p_{1}], [p_{2}], \ldots,
  [p_{n}]\bigr)$ is the colimit of the $\bovercat{\bD}{([p_{1}],
    [p_{2}], \ldots, [p_{n}])}$-diagram of simplicial sets that takes
  the object
  \begin{displaymath}
    (\alpha_{1}, \alpha_{2}, \ldots, \alpha_{n}) \colon
    \bigl([k], [k], \ldots, [k]\bigr) \longrightarrow
    \bigl([p_{1}], [p_{2}], \ldots, [p_{n}]\bigr)
  \end{displaymath}
  of $\bovercat{\bD}{([p_{1}], [p_{2}], \ldots, [p_{n}])}$ to the
  standard $k$-simplex $\Delta[k]$.  That object is the product for $1
  \le i \le n$ of morphisms $\alpha_{i}\colon [k] \to [p_{i}]$ in
  $\bD$, i.e., the product for $1 \le i \le n$ of a $k$-simplex of
  $\Delta[p_{i}]$, i.e., a $k$-simplex of $\Delta[p_{1}] \times
  \Delta[p_{2}] \times \cdots \times \Delta[p_{n}]$.

  Thus, $\LKan\Delta\bigl([p_{1}], [p_{2}], \ldots, [p_{n}]\bigr)$ is
  the colimit of the diagram indexed by the category of simplices of
  $\Delta[p_{1}] \times \Delta[p_{2}] \times \cdots \times
  \Delta[p_{n}]$ (see \lemref{lem:ColimCtSmp}) that takes each
  $k$-simplex of $\Delta[p_{1}] \times \Delta[p_{2}] \times \cdots
  \times \Delta[p_{n}]$ to the standard $k$-simplex $\Delta[k]$, and
  so \lemref{lem:ColimCtSmp} implies that $\LKan\Delta\bigl([p_{1}],
  [p_{2}], \ldots, [p_{n}]\bigr) \iso \Delta[p_{1}] \times
  \Delta[p_{2}] \times \cdots \times \Delta[p_{n}]$.

  For the natural transformation $\alpha$, note that a map of
  simplicial sets with domain $\Delta[k]$ is entirely determined by
  what it does to the nondegenerate $k$-simplex of $\Delta[k]$, which
  is the identity map of $[k]$.  The natural transformation
  $\alpha\colon \Delta \to \diagon\Delta^{(n)}$ on the object $[k]$ of
  $\bD$ takes that nondegenerate $k$-simplex of $\Delta[k]$ to the
  $k$-simplex of $\Delta[k]\times \Delta[k]\times \cdots \times
  \Delta[k]$ that is the image under the diagonal embedding of the
  identity map of $[k]$, which is the map $[k] \to \bigl([k], [k],
  \ldots, [k]\bigr)$ whose projection onto each factor is the identity
  map of $[k]$, i.e., the product of the nondegenerate $k$-simplices
  of each factor.
\end{proof}

%--------------------------------------------------------------------
\subsection{Quillen functors, but not Quillen equivalences}
\label{sec:NotEq}

Let $\cat M$ be a model category.  We show in \thmref{thm:NotQE} and
\corref{cor:notqeTop} that the right Quillen functor $\diagon\colon
\cat M^{\bD^{n}} \to \cat M^{\bD}$ (see \thmref{thm:RiQuil}) and its
left adjoint $\LKan\colon \cat M^{\bD} \to \cat M^{\bD^{n}}$ are not
Quillen equivalences when $\cat M$ is either the model category of
simplicial sets or the model category of topological spaces (see
\cite{MCATL}*{Def.~8.5.20}).

\begin{thm}
  \label{thm:NotQE}
  If $\cat M = \SiS$, the model category of simplicial sets, and $n
  \ge 2$, then the Quillen functors $\diagon\colon \cat M^{\bD^{n}}
  \to \cat M^{\bD}$ and its left adjoint $\LKan\colon \cat M^{\bD} \to
  \cat M^{\bD^{n}}$ are not Quillen equivalences.
\end{thm}

\begin{proof}
  We will discuss the case $n = 2$; the other cases are similar.  We
  will construct a cofibrant cosimplicial simplicial set $\diag X$, a
  fibrant bicosimplicial simplicial set Y, and a weak equivalence
  $f\colon \diag X \to \diagon\diag Y$ such that the corresponding map
  $\LKan\diag X \to \diag Y$ is not a weak equivalence.

  For each $k \ge 0$, let $(\Delta[k])^{0}$ denote the $0$-skeleton of
  $\Delta[k]$, and let $\diag X$ be the cosimplicial simplicial set
  that is the degreewise $0$-skeleton of the cosimplicial standard
  simplex $\Delta$, so that $\diag X^{k} = (\Delta[k])^{0}$.  Since
  the maximal augmentation of $\diag X$ is empty, $\diag X$ is
  cofibrant (see \cite{MCATL}*{Cor.~15.9.10}).

  Let $\diag W$ be the bicosimplicial simplicial set obtained from
  $\diag X$ by making it constant in the second index, i.e., $\diag
  W^{(p,q)} = \diag X^{p} = (\Delta[p])^{0}$, and let $\diag W \to
  \diag Y$ be a fibrant approximation to $\diag W$, so that $\diag W
  \to \diag Y$ is a weak equivalence of bicosimplicial simplicial sets
  and $\diag Y$ is fibrant.  There is an obvious isomorphism of
  cosimplicial simplicial sets $\diag X \to \diagon\diag W$, and our
  map $f\colon \diag X \to \diagon\diag Y$ is the composition $\diag X
  \to \diagon\diag W \to \diagon\diag Y$; it is the composition of an
  isomorphism and an objectwise weak equivalence, and so it is an
  objectwise weak equivalence, i.e., a weak equivalence of
  cosimplicial simplicial sets.

  The functor $\LKan\colon \SiS^{\bD} \to \SiS^{\bD^{2}}$ takes the
  cosimplicial standard simplex $\Delta$ to the bicosimplicial product
  of standard simplices $\Delta^{(2)}$ (see \thmref{thm:LeftKan}).
  Since the colimit of a diagram of simplicial sets is constructed
  degreewise, and every simplicial set in the diagram whose colimit is
  $(\LKan\diag X)^{(p,q)}$ is discrete (i.e., has all face and
  degeneracy operators isomorphisms), each $(\LKan\diag X)^{(p,q)}$ is
  also discrete, and so $\LKan\diag X = (\Delta^{(2)})^{0}$, the
  degreewise $0$-skeleton of $\Delta^{(2)}$.  Thus,
  $(\LKan\diagon\diag X)^{(1,1)}$ is the $0$-skeleton of
  $\Delta[1]\times \Delta[1]$, and has four path components, while
  $\diag Y^{(1,1)}$ is weakly equivalent to $(\Delta[1])^{0}$, and has
  two path components.  Thus, the map $\LKan\diag X \to \diag Y$ is
  not a weak equivalence.
\end{proof}

\begin{cor}
  \label{cor:notqeTop}
  If $\cat M = \Top$, the model category of topological spaces, and $n
  \ge 2$, then the Quillen functors $\diagon\colon \cat M^{\bD^{n}}
  \to \cat M^{\bD}$ and its left adjoint $\LKan\colon \cat M^{\bD} \to
  \cat M^{\bD^{n}}$ are not Quillen equivalences.
\end{cor}

\begin{proof}
  The geometric realization of the example in \thmref{thm:NotQE} is a
  cofibrant cosimplicial space $\diag X$, a fibrant bisimplicial space
  $\diag Y$, and a weak equivalence $f\colon \diag X \to \diag Y$ such
  that the corresponding map $\LKan\diag X \to \diag Y$ is not a weak
  equivalence.
\end{proof}

%--------------------------------------------------------------------
%--------------------------------------------------------------------

\section{Proof of \propref{prop:matchmatch}}
\label{sec:Proof}

Since $\inv{\bD^{n}} = \inv{\bD}\times \inv{\bD} \times \cdots \times
\inv{\bD}$ (see \cite{MCATL}*{Prop.~15.1.6}), the matching category
$\matchcat{\bD^{n}}{([k],[k],\ldots,[k])}$ has as objects
the maps
\begin{displaymath}
  \bigl([k] \to [p_{1}], [k] \to [p_{2}], \ldots, [k]\to [p_{n}]\bigr)
\end{displaymath}
in $\bD^{n}$ such that each $[k] \to [p_{i}]$ is a surjection and such
that at least one of them is not the identity map.  The diagonal
embedding of $\bD$ into $\bD^{n}$ (see \defref{def:diag}) takes the
matching category $\matchcat{\bD}{[k]}$ to the full subcategory of
$\matchcat{\bD^{n}}{([k],[k],\ldots,[k])}$ with objects the maps
\begin{displaymath}
  \{\phi^{n} \suchthat \text{$\phi\colon [k] \to [p]$ is a surjection
    and is not the identity map}\}
\end{displaymath}
and we will identify $\matchcat{\bD}{[k]}$ with its image in
$\matchcat{\bD^{n}}{([k],[k],\ldots,[k])}$.
Our map
\begin{multline*}
  P_{k}^{\bD^{n}} \quad=\quad
  \pullback
  {\bigl(\lim_{\matchcat{\bD^{n}}{([k],[k],\ldots,[k])}} \diag X\bigr)}
  {\bigl(\lim_{\matchcat{\bD^{n}}{([k],[k],\ldots,[k])}} \diag Y\bigr)}
  {\diag Y^{(k,k,\ldots,k)}}\\
  \longrightarrow\qquad
  P_{k}^{\bD} \quad=\quad
  \pullback{\bigl(\lim_{\matchcat{\bD}{[k]}} \diag X\bigr)}
  {\bigl(\lim_{\matchcat{\bD}{[k]}} \diag Y\bigr)}
  {\diag Y^{(k,k,\ldots,k)}}
\end{multline*}
(see \defref{def:Pdelta}) is induced by restricting the functors
$\diag X$ and $\diag Y$ to this subcategory.  We will define a nested
sequence of subcategories of
$\matchcat{\bD^{n}}{([k],[k],\ldots,[k])}$
\begin{displaymath}
  \matchcat{\bD}{[k]} = \cat C_{-1} \subset \cat C_{0} \subset \cdots
  \subset \cat C_{nk-1} = \matchcat{\bD^{n}}{([k],[k],\ldots,[k])}
\end{displaymath}
and for $-1 \le i \le nk-1$ we will let $P_{i}$ be the pullback
\begin{displaymath}
  \xymatrix{
    {P_{i}} \ar@{..>}[r] \ar@{.>}[d]
  & {\diag Y^{(k,k,\ldots,k)}} \ar[d]\\
  {\lim_{\cat C_{i}}\diag X} \ar[r]
  & {\lim_{\cat C_{i}}\diag Y \Period}
  }% xymatrix
\end{displaymath}
Thus, we will have a factorization of our map $P_{k}^{\bD^{n}} \to
P_{k}^{\bD}$ as
\begin{displaymath}
  P_{k}^{\bD^{n}} = P_{nk-1} \longrightarrow
  P_{nk-2} \longrightarrow \cdots \longrightarrow
  P_{-1} = P_{k}^{\bD}
\end{displaymath}
and we will show that the map $P_{i+1} \to P_{i} $ is a fibration for
$-1 \le i \le nk-2$.
\begin{defn}
  \label{def:SeqCat}
  If $n$ is a positive integer, $k$ is a nonnegative integer, and $-1
  \le i \le nk-1$, we let $\cat C_{i}$ be the full subcategory of
  $\matchcat{\bD^{n}}{([k],[k],\ldots,[k])}$ with objects the union of
  \begin{itemize}
  \item the objects of $\matchcat{\bD^{n}}{([k],[k],\ldots,[k])}$
    whose target is of degree at most $i$, and
  \item the objects of $\matchcat{\bD^{n}}{([k],[k],\ldots,[k])}$ in
    the image of the embedding of $\matchcat{\bD}{[k]}$.
    \end{itemize}
    That is, we let $\cat C_{i}$ be the full subcategory of
    $\matchcat{\bD^{n}}{([k],[k],\ldots,[k])}$ with objects the maps
  \begin{displaymath}
    (\phi_{1},\phi_{2},\ldots, \phi_{n})\colon
    \bigl([k],[k],\ldots, [k]\bigr) \longrightarrow
    \bigl([p_{1}], [p_{2}], \ldots, [p_{n}]\bigr)
  \end{displaymath}
  such that either $p_{1} + p_{2} + \cdots + p_{n} \le i$ or $\phi_{1}
  = \phi_{2} = \cdots = \phi_{n}$.
\end{defn}

\begin{prop}
  \label{prop:EachFib}
  If $-1 \le i \le nk-2$, then the map $P_{i+1} \to P_{i}$ is a
  fibration.
\end{prop}

\begin{proof}
  The objects of $\cat C_{i+1}$ that aren't in $\cat C_{i}$ are maps
  $\bigl([k] \to [p_{1}], [k] \to [p_{2}], \ldots, [k] \to
  [p_{n}]\bigr)$ such that $p_{1} + p_{2} + \cdots + p_{n} = i+1$
  (though not necessarily all such maps), and this set of maps can be
  divided into two subsets:
  \begin{itemize}
  \item the set $S_{i+1}$ of maps for which there exists an
    epimorphism $\psi\colon [k] \to [j]$ with $j<k$ and a
    factorization through $\psi^{n}\colon \bigl([k],[k],\ldots,
    [k]\bigr) \to \bigl([j],[j],\ldots, [j]\bigr)$,
  \item the set $T_{i+1}$ of maps for which there is no such
    factorization.
  \end{itemize}
  We let $\cat C'_{i+1}$ be the full subcategory of
  $\matchcat{\bD^{n}}{([k],[k],\ldots,[k])}$ with objects the union of
  $S_{i+1}$ with the objects of $\cat C_{i}$, and define $P'_{i+1}$ as
  the pullback
  \begin{displaymath}
    \xymatrix{
      {P'_{i+1}} \ar@{..>}[r] \ar@{..>}[d]
      & {\diag Y^{(k,k,\ldots,k)}} \ar[d]\\
      {\lim_{\cat C'_{i+1}}\diag X} \ar[r]
      & {\lim_{\cat C'_{i+1}}\diag Y \Period}
    }% xymatrix
  \end{displaymath}
  We have inclusions of categories $\cat C_{i} \subset \cat C'_{i+1}
  \subset \cat C_{i+1}$, and the maps
  \begin{displaymath}
    \lim_{\cat C_{i+1}} \diag X \longrightarrow
    \lim_{\cat C_{i}}\diag X
    \qquad\text{and}\qquad
    \lim_{\cat C_{i+1}} \diag Y \longrightarrow
    \lim_{\cat C_{i}}\diag Y
  \end{displaymath}
  factor as
  \begin{displaymath}
    \lim_{\cat C_{i+1}} \diag X \longrightarrow
    \lim_{\cat C'_{i+1}} \diag X \longrightarrow
    \lim_{\cat C_{i}}\diag X
    \qquad\text{and}\qquad
    \lim_{\cat C_{i+1}} \diag Y \longrightarrow
    \lim_{\cat C'_{i+1}} \diag Y \longrightarrow
    \lim_{\cat C_{i}}\diag Y \Period
  \end{displaymath}
  These factorizations induce a factorization
  \begin{displaymath}
    P_{i+1} \longrightarrow  P'_{i+1} \longrightarrow P_{i}
    \qquad\text{of the map}\qquad
    P_{i+1} \longrightarrow P_{i} \Period
  \end{displaymath}
  \propref{prop:isoprime} asserts that the map $P'_{i+1} \to P_{i}$ is
  an isomorphism and \propref{prop:IndFibr} asserts that the map
  $P_{i+1} \to P'_{i+1}$ is a fibration.
\end{proof}

\begin{prop}
  \label{prop:isoprime}
  For $-1 \le i \le nk-2$, the map $P'_{i+1} \to P_{i}$ is an
  isomorphism.
\end{prop}

The proof of \propref{prop:isoprime} is in \secref{sec:prfisoprime}.

\begin{prop}
  \label{prop:IndFibr}
  For $-1 \le i \le nk-1$, the map $P_{i+1} \to P'_{i+1}$ is a
  fibration.
\end{prop}

The proof of \propref{prop:IndFibr} is in \secref{sec:propIndFibr}.

%--------------------------------------------------------------------
%--------------------------------------------------------------------
\section{Proof of \propref{prop:isoprime}}
\label{sec:prfisoprime}

\begin{lem}
  \label{lem:TermFact}
  Every morphism
  \begin{displaymath}
    (\alpha_{1}, \alpha_{2}, \ldots, \alpha_{n}) \colon
    \bigl([k], [k], \ldots, [k]\bigr) \longrightarrow
    \bigl([p_{1}], [p_{2}], \ldots, [p_{n}]\bigr)
  \end{displaymath}
  in $\inv{\bD^{n}}$ with domain a diagonal object has a terminal
  factorization through a diagonal morphism, i.e., an epimorphism
  $\beta\colon [k] \to [q]$ and a factorization
  \begin{displaymath}
    \bigl([k], [k], \ldots, [k]\bigr)
    \xrightarrow{\enskip\beta^{n}\enskip}
    \bigl([q], [q], \ldots, [q]\bigr) \longrightarrow
    \bigl([p_{1}], [p_{2}], \ldots, [p_{n}]\bigr)
  \end{displaymath}
  of $(\alpha_{1}, \alpha_{2}, \ldots, \alpha_{n})$ such that every
  epimorphism $\gamma\colon [k] \to [r]$ and factorization
  \begin{displaymath}
    \bigl([k], [k], \ldots, [k]\bigr)
    \xrightarrow{\enskip\gamma^{n}\enskip}
    \bigl([r], [r], \ldots, [r]\bigr) \longrightarrow
    \bigl([p_{1}], [p_{2}], \ldots, [p_{n}]\bigr)
  \end{displaymath}
  of $(\alpha_{1}, \alpha_{2}, \ldots, \alpha_{n})$ through a diagonal
  morphism of $\bD^{n}$ factors uniquely as
  \begin{displaymath}
    \bigl([k], [k], \ldots, [k]\bigr) \xrightarrow{\gamma^{n}}
    \bigl([r], [r], \ldots, [r]\bigr) \xrightarrow{\delta^{n}}
    \bigl([q], [q], \ldots, [q]\bigr) \longrightarrow
    \bigl([p_{1}], [p_{2}], \ldots, [p_{n}]\bigr)
  \end{displaymath}
  with $\delta\gamma = \beta$.
\end{lem}

\begin{proof}
  Each of the epimorphisms $\alpha_{j}\colon [k] \to [p_{j}]$ is
  determined by the set $U_{j}$ of integers $i$ such that
  $\alpha_{j}(i) = \alpha_{j}(i+1)$.  We let $U = \bigcap_{1 \le j \le
    n} U_{j}$.  The set $U$ now determines an epimorphism $\beta\colon
  [k] \to [q]$ for some $q \le k$, and the terminal factorization of
  $\alpha$ is the factorization through $\beta^{n}\colon \bigl([k],
  [k],\ldots,[k]\bigr) \to \bigl([q], [q], \ldots, [q]\bigr)$.
\end{proof}

\begin{prop}
  \label{prop:cofinal}
  For $-1 \le i \le nk-2$, the inclusion of categories $\cat C_{i}
  \subset \cat C'_{i+1}$ is left cofinal (see
  \cite{MCATL}*{Def.~14.2.1}).
\end{prop}

\begin{proof}
  Let $\alpha = \bigl([k] \to [p_{1}], [k] \to [p_{2}], \ldots, [k]
  \to [p_{n}]\bigr)$ be an object of $\cat C'_{i+1}$ that isn't in
  $\cat C_{i}$.  Since every morphism in $\inv{\bD^{n}}$ lowers
  degree, the only objects of $\overcat{\cat C_{i}}{\alpha}$ are
  factorizations of $\alpha$ through $\phi^{n}\colon
  \bigl([k],[k],\ldots, [k]\bigr) \to \bigl([j],[j],\ldots, [j]\bigr)$
  for some epimorphism $\phi\colon [k] \to [j]$ with $j < k$, and
  \lemref{lem:TermFact} implies that there is a terminal such
  factorization, i.e., one through which all other factorizations
  factor uniquely.  That terminal factorization is a terminal object
  of the overcategory $\overcat{\cat C_{i}}{\alpha}$, and so that
  overcategory is nonempty and connected, and so the inclusion $\cat
  C_{i} \subset \cat C'_{i+1}$ is left cofinal.
\end{proof}

\begin{proof}[Proof of \propref{prop:isoprime}]
  For $-1 \le i \le nk-2$, \propref{prop:cofinal} implies that the
  inclusion of categories $\cat C_{i} \subset \cat C'_{i+1}$ is left
  cofinal, and so the maps $\lim_{\cat C'_{i+1}} \diag X \to
  \lim_{\cat C_{i}}\diag X$ and $\lim_{\cat C'_{i+1}} \diag Y \to
  \lim_{\cat C_{i}}\diag Y$ are isomorphisms (see
  \cite{MCATL}*{Thm.~14.2.5}), and so the induced map $P'_{i+1} \to
  P_{i}$ is an isomorphism.
\end{proof}

%--------------------------------------------------------------------
%--------------------------------------------------------------------
\section{Proof of \propref{prop:IndFibr}}
\label{sec:propIndFibr}

\begin{notation}
  \label{not:obUC}
  To save space, we will use the following compact notation:
  \begin{itemize}
  \item We will use $\obDn$ to denote an object $\bigl([p_{1}],
    [p_{2}], \ldots, [p_{n}]\bigr)$ of $\bD^{n}$.
  \item If $k \ge 0$, we will use $\obUC$ to denote an object
    $\bigl([k],[k], \ldots, [k]\bigr) \to \bigl([p_{1}], [p_{2}],
    \ldots, [p_{n}]\bigr)$ of the matching category
    $\matchcat{\bD^{n}}{([k], [k], \ldots, [k])}$.
  \end{itemize}
\end{notation}

\begin{lem}
  \label{lem:pbCtoCprime}
  For each $n$-cosimplicial object $\diag X$ in $\cat M$ there is a
  pullback square
  \begin{equation}
    \label{diag:pbCtoCprime}
    \vcenter{
      \xymatrix{
        {\lim_{\cat C_{i+1}} \diag X} \ar[r] \ar[d]
        & {\lim_{\cat C'_{i+1}} \diag X} \ar[d]\\
        {\prod_{(\obUC)\in T_{i+1}} \hspace{-2em}\diag X_{\obDn}} \ar[r]
        & {\prod_{(\obUC)\in T_{i+1}}
          \lim_{\matchcat{\bD^{n}}{\obDn}} \diag X \Period}
      }% xymatrix
    }% vcenter
  \end{equation}
\end{lem}

\begin{proof}
  For every element $\obUC$ of $T_{i+1}$, every object of the matching
  category $\matchcat{\bD^{n}}{\obDn}$ is a map to an object of degree
  at most $i$, and so there is a functor $\matchcat{\bD^{n}}{\obDn}
  \to \cat C'_{i+1}$ that takes the object $\bigl([p_{1}], [p_{2}],
  \ldots, [p_{n}]\bigr) \to \bigl([q_{1}], [q_{2}], \ldots,
  [q_{n}]\bigr)$ to the composition $\bigl([k], [k], \ldots, [k]\bigr)
  \to \bigl([p_{1}], [p_{2}], \ldots, [p_{n}]\bigr) \to \bigl([q_{1}],
  [q_{2}], \ldots, [q_{n}]\bigr)$; this induces a map $\lim_{\cat
    C'_{i+1}} \to \lim_{\matchcat{\bD^{n}}{\obDn}} \diag X$ that is
  the projection of the right hand vertical map onto the factor
  indexed by $\obUC$.  We thus have a commutative square as in
  \diagref{diag:pbCtoCprime}.

  The objects of $\cat C_{i+1}$ are the objects of $\cat C'_{i+1}$
  together with the elements of $T_{i+1}$, and so a map to $\lim_{\cat
    C_{i+1}} \diag X$ is determined by a map to $\lim_{\cat C'_{i+1}}
  \diag X$ and a map to $\prod_{(\obUC)\in T_{i+1}} \diag X_{\obDn}$.
  Since there are no non-identity morphisms in $\cat C_{i+1}$ with
  codomain an element of $T_{i+1}$, and the only non-identity
  morphisms with domain an element $\obUC$ of $T_{i+1}$ are the
  objects of the matching category $\matchcat{\bD^{n}}{\obDn}$, maps
  to $\lim_{\cat C'_{i+1}} \diag X$ and to $\prod_{(\obUC)\in T_{i+1}}
  X_{\obDn}$ determine a map to $\lim_{\cat C_{i+1}} \diag X$ if and
  only if their compositions to $\prod_{(\obUC)\in T_{i+1}}
  \lim_{\matchcat{\bD^{n}}{\obDn}} \diag X$ agree.  Thus, the diagram
  is a pullback square.
\end{proof}

Define $Q$ and $R$ by letting the squares
\begin{equation}
  \label{diag:pullbacks}
  \vcenter{
    \xymatrix{
      {Q} \ar@{..>}[r] \ar@{..>}[d]
      & {\lim_{\cat C'_{i+1}} \diag X} \ar[d]\\
      {\lim_{\cat C_{i+1}} \diag Y} \ar[r]
      & {\lim_{\cat C'_{i+1}} \diag Y}
    }% xymatrix
  }% vcenter
  \qquad\text{and}\quad
  \vcenter{
    \xymatrix{
      {R} \ar@{..>}[r] \ar@{..>}[d]
      & {\prod_{(\obUC)\in T_{i+1}}
        \lim_{\matchcat{\bD^{n}}{\obDn}} \diag X} \ar[d]\\
      {\prod_{(\obUC)\in T_{i+1}}
        \hspace{-2em}\diag Y_{\obDn}} \ar[r]
      & {\prod_{(\obUC)\in T_{i+1}}
        \lim_{\matchcat{\bD^{n}}{\obDn}} \diag Y}
    }% xymatrix
  }%vcenter
\end{equation}
be pullbacks, and consider the commutative diagram
\begin{equation}
  \label{diag:bigcube}
  \vcenter{
    \xymatrix@=.8em@R=4ex{
      {\lim_{\cat C_{i+1}} \diag X} \ar[rrr]^{s} \ar[drr]^(.7){a}
      \ar[ddr]_{\delta} \ar[ddd]_{u}
      &&& {\lim_{\cat C'_{i+1}} \diag X} \ar[ddr]^{\beta}
      \ar'[dd][ddd]^{v}\\
      && {Q} \ar[ur]_{c} \ar[dl]_{d} \ar'[d][ddd]^(.3){g}\\
      & {\lim_{\cat C_{i+1}}\diag Y} \ar[rrr]^(.7){s'}
      \ar[ddd]_(.75){u'}
      &&& {\lim_{\cat C'_{i+1}}\diag Y} \ar[ddd]^{v'}\\
      {\prod_{(\obUC)\in T_{i+1}}
        \hspace{-2em}\diag X_{\obDn}} \ar[ddr]_{\gamma}
      \ar[drr]|!{[ru];[rdd]}{\hole}^(.75){b}
      \ar[rrr]|!{[ruu];[rd]}{\hole}^{t}
      |!{[uurr];[rrd]}{\hole}
      &&& {\hspace{-1em}\prod_{\substack{\phantom{X}\\(\obUC)\in T_{i+1}}}
        \hspace{-2em}\lim_{\matchcat{\bD^{n}}{\obDn}}\diag X} \ar[ddr]\\
      && {R} \ar[ur]_-{e}  \ar[dl]^(.3){f}\\
      & {\hspace{-2em}\prod_{(\obUC)\in T_{i+1}}
        \hspace{-2em}\diag Y_{\obDn}} \ar[rrr]_-{t'}
      &&& {\hspace{-2em}\prod_{\substack{\phantom{X}\\(\obUC)\in T_{i+1}}}
        \hspace{-2em}\lim_{\matchcat{\bD^{n}}{\obDn}}\diag Y}
    }% xymatrix
  }% vcenter
\end{equation}
\lemref{lem:pbCtoCprime} implies that the front and back rectangles
are pullbacks.

\begin{lem}
  \label{lem:PBf}
  The square
  \begin{equation}
    \label{diag:PBf}
    \vcenter{
      \xymatrix{
        {\lim_{\cat C_{i+1}}\diag X} \ar[r]^-{a} \ar[d]_{u}
        & {Q} \ar[d]^{g}\\
        {\prod_{(\obUC)\in T_{i+1}}
          \hspace{-2em}\diag X_{\obDn}} \ar[r]_-{b}
        & {R}
      }% xymatrix
    }% vcenter
  \end{equation}
  is a pullback.
\end{lem}

\begin{proof}
  Let $W$ be an object of $\cat M$ and let $h\colon W \to
  \prod_{(\obUC)\in T_{i+1}} \diag X_{\obDn}$ and $k\colon W \to Q$ be
  maps such that $gk = bh$; we will show that there is a unique map
  $\phi\colon W \to \lim_{\cat C_{i+1}} \diag X$ such that $a\phi = k$
  and $u\phi = h$.
  \begin{displaymath}
    \xymatrix{
      {W} \ar@/^3ex/[drr]^{k} \ar@/_3ex/[ddr]_{h}
      \ar@{..>}[dr]^{\phi}\\
      &{\lim_{\cat C_{i+1}} \diag X} \ar[r]^-{a} \ar[d]_{u}
      & {Q} \ar[d]^{g}\\
      &{\prod_{(\obUC)\in T_{i+1}}
        \hspace{-2em}\diag X_{\obDn}} \ar[r]_-{b}
      & {R}
    }% xymatrix
  \end{displaymath}
  The map $ck\colon W \to \lim_{\cat C'_{i+1}} \diag X$ has the
  property that $v(ck) = egk = ebh = th$, and since the back rectangle
  of \diagref{diag:bigcube} is a pullback, the maps $ck$ and $h$
  induce a map $\phi\colon W \to \lim_{\cat C_{i+1}} \diag X$ such
  that $u\phi = h$ and $s\phi = ck$.  We must show that $a\phi = k$,
  and since $Q$ is a pullback as in \diagref{diag:pullbacks}, this is
  equivalent to showing that $ca\phi = ck$ and $da\phi = dk$.

  Since $ck = s\phi = ca\phi$, we need only show that $da\phi = dk$.
  Since the front rectangle of \diagref{diag:bigcube} is a pullback,
  it is sufficient to show that $s'da\phi = s'dk$ and $u'da\phi =
  u'dk$.  For the first of those, we have
  \begin{displaymath}
    s'da\phi = s'\delta\phi = \beta s\phi = \beta ck = s'dk\\
  \end{displaymath}
  and for the second, we have
  \begin{displaymath}
    u'da\phi = u'\delta\phi = \gamma u\phi = fbu\phi = fbh = fgk =
    u'dk
  \end{displaymath}
  and so the map $\phi$ satisfies $a\phi = k$ and $u\phi = h$.

  To see that $\phi$ is the unique such map, let $\psi\colon W \to
  \lim_{\cat C_{i+1}} \diag X$ be another map such that $a\psi = k$
  and $u\psi = h$.  We will show that $s\psi = s\phi$ and $u\psi =
  u\phi$; since the back rectangle of \diagref{diag:bigcube} is a
  pullback, this will imply that $\psi = \phi$.

  Since $u\psi = h = u\phi$, we need only show that $s\psi = s\phi$,
  which follows because $s\psi = ca\psi = ck = s\phi$.
\end{proof}

\begin{lem}
  \label{lem:lastfib}
  If $\diag X \to \diag Y$ is a fibration of $n$-cosimplicial objects,
  then the natural map
  \begin{displaymath}
    \lim_{\cat C_{i+1}} \diag X \longrightarrow
    Q = \pullback{\lim_{\cat C'_{i+1}}\diag X}
    {(\lim_{\cat C'_{i+1}}\diag Y)}{\lim_{\cat C_{i+1}}\diag Y}
  \end{displaymath}
  is a fibration.
\end{lem}

\begin{proof}
  \lemref{lem:PBf} gives us the pullback square in \diagref{diag:PBf},
  where $Q$ and $R$ are defined by the pullbacks in
  \diagref{diag:pullbacks}.  Since $\diag X \to \diag Y$ is a
  fibration of $n$-cosimplicial objects, the map $\prod_{(\obUC)\in
    T_{i+1}} \diag X_{\obDn} \to R$ is a product of fibrations and is
  thus a fibration, and so the map $\lim_{\cat C_{i+1}}\diag X \to Q =
  \pullback{\lim_{\cat C'_{i+1}}\diag X} {(\lim_{\cat C'_{i+1}}\diag
    Y)}{\lim_{\cat C_{i+1}}\diag Y}$ is a pullback of a fibration and
  is thus a fibration.
\end{proof}

\begin{lem}[Reedy]
  \label{lem:reedy}
  If both the front and back squares in the diagram
  \begin{displaymath}
    \xymatrix@=1.5em{
      {A} \ar[rr] \ar[dr]^{f_{A}} \ar[dd]
      && {B} \ar[dr]^{f_{B}} \ar'[d][dd]\\
      & {A'} \ar[rr] \ar[dd]
      && {B'} \ar[dd]\\
      {C} \ar[dr]_{f_{C}} \ar'[r][rr]
      && {D} \ar[dr]^{f_{D}}\\
      & {C'} \ar[rr]
      && {D'}
    }
  \end{displaymath}
  are pullbacks and both $f_{B}\colon B \to B'$ and $C \to
  \pullback{C'}{D'}{D}$ are fibrations, then $f_{A}\colon A \to A'$ is
  a fibration.
\end{lem}

\begin{proof}
  This is the dual of a lemma of Reedy (see
  \cite{MCATL}*{Lem.~7.2.15 and Rem.~7.1.10}).
\end{proof}

\begin{proof}[Proof of \propref{prop:IndFibr}]
  We have a commutative diagram
  \begin{displaymath}
    \xymatrix@=.25em{
      {P_{i+1}} \ar[rr] \ar[dr] \ar[dd]
      && {Y^{(k,k,\ldots,k)}} \ar[dr] \ar'[d][dd]\\
      & {P'_{i+1}} \ar[rr] \ar[dd]
      && {Y^{(k,k,\ldots,k)}} \ar[dd]\\
      {\lim_{\cat C_{i+1}}\diag X} \ar[dr] \ar'[r][rr]
      && {\lim_{\cat C_{i+1}}\diag Y} \ar[dr]\\
      & {\lim_{\cat C'_{i+1}}\diag X} \ar[rr]
      && {\lim_{\cat C'_{i+1}}\diag Y}
    }
  \end{displaymath}
  in which the front and back squares are pullbacks, and so
  \lemref{lem:reedy} implies that it is sufficient to show that the
  map
  \begin{displaymath}
    \lim_{\cat C_{i+1}} \diag X \longrightarrow
    \pullback{\lim_{\cat C'_{i+1}}\diag X}
    {(\lim_{\cat C'_{i+1}}\diag Y)}{\lim_{\cat C_{i+1}}\diag Y}
  \end{displaymath}
  is a fibration; that is the statement of \lemref{lem:lastfib}.
\end{proof}

%--------------------------------------------------------------------
%--------------------------------------------------------------------
\section{Frames, homotopy cotensors, and homotopy function complexes}
\label{sec:frames}

If $\diag X$ is a cosimplicial simplicial set, then both the total
space $\Tot\diag X$ and the homotopy limit $\holim\diag X$ are built
from simplicial sets of the form $(\diag X^{n})^{K}$ for $n \ge 0$ and
$K$ a simplicial set, where $(\diag X^{n})^{K}$ is the simplicial set
of maps $K \to \diag X^{n}$.  If $\cat M$ is a simplicial model
category and $\diag X$ is a cosimplicial object in $\cat M$, then the
simplicial model category structure includes objects $(\diag
X^{n})^{K}$ of $\cat M$, called the \emph{cotensor} of $\diag X^{n}$
and $K$, and $\Tot\diag X$ and $\holim\diag X$ are built from those.
When $\cat M$ is a (possibly non-simplicial) model category and $\diag
X$ is a cosimplicial object in $\cat M$, we need objects of $\cat M$
that can play the role of the $(\diag X^{n})^{K}$, and to define those
we choose a \emph{simplicial frame} on the objects of $\cat M$.  A
simplicial frame on an object $X$ of $\cat M$ is a simplicial object
$\diXhat$ of $\cat M$ such that
\begin{itemize}
\item $\diXhat_{0}$ is isomorphic to $X$,
\item all the face and degeneracy operators of $\diXhat$ are weak
  equivalences, and
\item if $X$ is a fibrant object of $\cat M$ then $\diXhat$ is a Reedy
  fibrant simplicial object.
\end{itemize}
Given such a simplicial frame on $X$, we use $\diXhat_{n}$ to play the
role of $X^{\Delta[n]}$, and if $K$ is a simplicial set we construct
the \emph{homotopy cotensor} $\diXhat^{K}$ of $X$ and $K$ as the limit
of a diagram of the $\diXhat_{n}$ indexed by the opposite of the
category of simplices of $K$ (see \defref{def:realization}).

%--------------------------------------------------------------------
\subsection{Frames}

\begin{defn}[Frames and Reedy frames]
  \label{def:frame}
  \leavevmode
  \begin{enumerate}
  \item If $\cat M$ is a model category and $X$ is an object of $\cat
    M$, then a \emph{simplicial frame} on $X$ is a simplicial object
    $\diXhat$ together with an isomorphism $X \iso \diXhat_{0}$ such
    that
    \begin{itemize}
    \item all the face and degeneracy operators of $\diXhat$ are weak
      equivalences and
    \item if $X$ is a fibrant object of $\cat M$ then $\diXhat$ is a
      Reedy fibrant simplicial object in $\cat M$.
    \end{itemize}
    Equivalently,a simplicial frame on $X$ is a simplicial object
    $\diXhat$ in $\cat M$ together with a weak equivalence $\cs X \to
    \diXhat$ (where $\cs X$ is the constant simplicial object on $X$)
    in the Reedy model category structure on $\cat M^{\bD\op}$ such
    that
    \begin{itemize}
    \item the induced map $X \to \diXhat_{0}$ is an isomorphism, and
    \item if $X$ is a fibrant object of $\cat M$, then $\diXhat$ is a
      Reedy fibrant simplicial object.
    \end{itemize}
    We often refer to $\diXhat$ as a simplicial frame on $X$, without
    explicitly mentioning the map $\cs X \to \diXhat$.

    If $\cat M$ is a simplicial model category and $X$ is an object of
    $\cat M$, the \emph{standard simplicial frame on $X$} is the
    simplicial object $\diXhat$ in which $\diXhat_{n} = X^{\Delta[n]}$
    (see \cite{MCATL}*{Prop.~16.6.4}).
  \item If $\cat M$ is a model category, $\cat C$ is a small category,
    and $\diag X$ is a $\cat C$-diagram in $\cat M$, then a
    \emph{simplicial frame on $\diag X$} is a $\cat C$-diagram
    $\diXhat$ of simplicial objects in $\cat M$ together with a map of
    diagrams $j\colon \cs \diag X \to \diXhat$ from the diagram of
    constant simplicial objects such that, for every object $\alpha$
    of $\cat C$, the map $j_{\alpha}\colon \cs\diag X_{\alpha} \to
    \diXhat_{\alpha}$ is a simplicial frame on $\diag X_{\alpha}$.

    If $\cat M$ is a simplicial model category, then the
    \emph{standard simplicial frame} on $\diag X$ is the frame
    $\diXhat$ on $\diag X$ such that $\diXhat_{\alpha}$ is the
    standard simplicial frame on $\diag X_{\alpha}$ for every object
    $\alpha$ of $\cat C$.
  \item If $\cat M$ is a model category, $\cat C$ is a Reedy category,
    and $\diag X$ is a $\cat C$-diagram in $\cat M$, then a
    \emph{Reedy simplicial frame on $\diag X$} is a simplicial frame
    $\diXhat$ on $\diag X$ such that if $\diag X$ is a Reedy fibrant
    $\cat C$-diagram in $\cat M$, then $\diXhat$ is a Reedy fibrant
    $\cat C$-diagram in $\cat M^{\bD\op}$.

    If $\cat M$ is a simplicial model category and $\cat C$ is a Reedy
    category, then for any $\cat C$-diagram $\diag X$ in $\cat M$ the
    standard simplicial frame $\diXhat$ on $\diag X$ (see
    \defref{def:frame}) is a Reedy simplicial frame (see
    \cite{MCATL}*{Prop.~16.7.9}).
  \end{enumerate}
\end{defn}

\begin{prop}[Existence and essential uniqueness of frames]
  \label{prop:frameunique}
  \leavevmode
  \begin{enumerate}
  \item If $\cat M$ is a model category and $X$ is an object of $\cat
    M$, then there exists a simplicial frame on $X$ and any two
    simplicial frames on $X$ are connected by an essentially unique
    zig-zag (see \cite{MCATL}*{Def.~14.4.2}) of weak equivalences of
    simplicial frames on $X$.
  \item If $\cat M$ is a model category, $\cat C$ is a small category,
    and $\diag X$ is a $\cat C$-diagram in $\cat M$, then there exists
    a simplicial frame on $\diag X$ and any two simplicial frames on
    $\diag X$ are connected by an essentially unique zig-zag (see
    \cite{MCATL}*{Def.~14.4.2}) of weak equivalences of simplicial
    frames on $\diag X$.
  \item If $\cat M$ is a model category, $\cat C$ is a Reedy category,
    and $\diag X$ is a $\cat C$-diagram in $\cat M$, then there exists
    a Reedy simplicial frame on $\diag X$ and any two Reedy simplicial
    frames on $\diag X$ are connected by an essentially unique zig-zag
    (see \cite{MCATL}*{Def.~14.4.2}) of weak equivalences of Reedy
    simplicial frames on $\diag X$.
  \item If $\cat M$ is a model category and $\cat C$ is a Reedy
    category, then there exists a functorial Reedy simplicial frame on
    every $\cat C$-diagram on $\cat M$.
  \end{enumerate}
\end{prop}

\begin{proof}
  See \cite{MCATL}*{Thm.~16.6.18, Thm.~16.7.6, Prop.~16.7.11, and
    Thm.~16.7.14}.
\end{proof}

%--------------------------------------------------------------------
\subsection{Homotopy cotensors}
\label{sec:cotensors}

\begin{defn}
  \label{def:realization}
  If $\cat M$ is a model category, $X$ is an object of $\cat M$,
  $\diXhat$ is a simplicial frame on $X$, and $K$ is a simplicial set,
  then the \emph{homotopy cotensor} $\diXhat^{K}$ is defined to be the
  object of $\cat M$ that is the limit of the $(\Delta K)\op$-diagram
  in $\cat M$ (see \lemref{lem:ColimCtSmp}) that takes the object
  $\Delta[n] \to K$ of $(\Delta K)\op = \overcat{\bD}{K}\op$ to
  $\diXhat_{n}$ and takes the commutative triangle
  \begin{displaymath}
    \xymatrix{
      {\Delta[n]} \ar[rr]^{\alpha} \ar[dr]
      && {\Delta[k]} \ar[dl]\\
      & {K}
    }
  \end{displaymath}
  to the map $\alpha^{*}\colon \diXhat_{k} \to \diXhat_{n}$ (see
  \cite{MCATL}*{Def.~16.3.1}).
\end{defn}

\begin{prop}
  \label{prop:SMCreal}
  If $\cat M$ is a simplicial model category, $X$ is an object of
  $\cat M$, $\diXhat$ is the standard simplicial frame on $X$ (see
  \defref{def:frame}), and $K$ is a simplicial set, then $\diXhat^{K}$
  is naturally isomorphic to $X^{K}$.
\end{prop}

\begin{proof}
  See \cite{MCATL}*{Prop.~16.6.6}.
\end{proof}

%--------------------------------------------------------------------
\subsection{Homotopy function complexes}
\label{sec:HoFncCmp}

Although homotopy function complexes have many important properties
(see \cite{MCATL}*{Chap.~17}), our only use for them here is the
adjointness result \thmref{thm:EndAdj}, which will be used in the
proofs of \thmref{thm:IsoTot} and \thmref{thm:holimdiag}.

\begin{defn}
  \label{def:MpSmpSt}
  Let $\cat M$ be a model category and let $W$ be an object of $\cat
  M$.
  \begin{enumerate}
  \item If $X$ is an object of $\cat M$ and $\diXhat$ is a simplicial
    frame on $X$, then $\map_{\diXhat}(W, X)$ will denote the
    simplicial set, natural in both $W$ and $\diXhat$, defined by
    \begin{displaymath}
      \map_{\diXhat}(W,X)_{n} = \cat M(W, \diXhat_{n})
    \end{displaymath}
    with face and degeneracy maps induced by those in $\diXhat$.  If
    $W$ is cofibrant and $X$ is fibrant, then $\map_{\diXhat}(W,X)$ is
    a right homotopy function complex from $W$ to $X$ (see
    \cite{MCATL}*{Def.~17.2.1}).
  \item If $\cat C$ is a small category, $\diag X$ is a $\cat
    C$-diagram in $\cat M$, and $\diXhat$ is a simplicial frame on
    $\diag X$, then $\map_{\diXhat}(W, \diag X)$ will denote the $\cat
    C$-diagram of simplicial sets that on an object $\alpha$ of $\cat
    C$ is the simplicial set $\map_{\diXhat}(W, \diag X_{\alpha})$.
  \end{enumerate}
\end{defn}

\begin{prop}
  \label{prop:MpSpc}
  If $\cat M$ is a simplicial model category, $W$ and $X$ are objects
  of $\cat M$, and $\diXhat$ is the standard simplicial frame on $X$,
  then $\map_{\diXhat}(W,X)$ is naturally isomorphic to $\Map(W,X)$,
  the simplicial set of maps that is part of the structure of the
  simplicial model category $\cat M$.
\end{prop}

\begin{proof}
  We have natural isomorphisms
  \begin{align*}
    \map_{\diXhat}(W,X)_{n} &= \cat M(W, \diXhat_{n})\\
    &= \cat M(W, X^{\Delta[n]})\\
    &\iso \SiS\bigl(\Delta[n], \Map(W,X)\bigr)\\
    &\iso \Map(W,X)_{n} \Period \qedhere
  \end{align*}
\end{proof}

\begin{prop}
  \label{prop:adjointness}
  Let $\cat M$ be a model category and let $W$ be an object of $\cat
  M$.  If $X$ is an object of $\cat M$, $\diXhat$ is a simplicial
  resolution of $X$, and $K$ is a simplicial set, then there is a
  natural isomorphism of sets
    \begin{displaymath}
      \SiS\bigl(K, \map_{\diXhat}(W, X)\bigr) \iso
      \cat M(W, \diXhat^{K}) \Period
    \end{displaymath}
\end{prop}

\begin{proof}
  Since $\diXhat^{K}$ is defined as a limit (see
  \defref{def:realization}), an element of $\cat M(W, \diXhat^{K})$ is
  a collection of maps $W \to \diXhat_{n}$, one for each $n$-simplex
  of $K$, that commute with the face and degeneracy operators.  This
  is also a description of an element of $\SiS\bigl(K,
  \map_{\diXhat}(W,X)\bigr)$ (see also \cite{MCATL}*{Thm.~16.4.2}).
\end{proof}

\begin{thm}
  \label{thm:EndAdj}
  Let $\cat M$ be a model category and let $\cat C$ be a
  small category.  If $\diag X$ is a $\cat C$-diagram in $\cat M$,
  $\diXhat$ is a simplicial frame on $\diag X$,
  $\diag K$ is a $\cat C$-diagram of simplicial sets, and $W$ is an
  object of $\cat M$, then there is a natural isomorphism of sets
  \begin{displaymath}
    \cat M\bigl(W, \hom_{\diXhat}^{\cat C}(\diag K, \diag X)\bigr)
    \iso
    \SiS^{\cat C}\bigl(\diag K, \map_{\diXhat}(W,\diag X)\bigr)
  \end{displaymath}
  (where $\map_{\diXhat}(W,\diag X)$ is as in \defref{def:MpSmpSt} and
  $\hom_{\diXhat}^{\cat C}(\diag K, \diag X)$ is the end of the
  functor $\diXhat^{\diag K}\colon \cat C\times \cat C\op \to \cat M$;
  see \cite{MCATL}*{Def.~19.2.2}).
\end{thm}

\begin{proof}
  The object $\hom_{\diXhat}^{\cat C}(\diag K, \diag X)$ is defined
  (see \cite{MCATL}*{Def.~19.2.2}) as the limit of the diagram
  \begin{displaymath}
    \xymatrix{
      {\prod_{\alpha\in\Ob(\cat C)}
        (\diXhat_\alpha)^{\diag K_{\alpha}}\quad}
      \ar[r]<.8ex>^-{\phi} \ar[r]<-.6ex>_-{\psi}
      &{\quad\prod_{(\sigma\colon \alpha \to \alpha') \in \cat C}
        (\diXhat_{\alpha'})^{\diag K_{\alpha}}}
    }
  \end{displaymath}
  where the projection of the map $\phi$ on the factor indexed by
  $\sigma\colon \alpha \to \alpha'$ is the composition of a natural
  projection from the product with the map
  \begin{displaymath}
    \sigma_{*}^{1_{\diag K_{\alpha}}}\colon
    (\diXhat_{\alpha})^{\diag K_{\alpha}}
    \longrightarrow
    (\diXhat_{\alpha'})^{\diag K_{\alpha}}    
  \end{displaymath}
  (where $\sigma_{*}\colon \diXhat_{\alpha} \to \diXhat_{\alpha'}$)
  and the projection of the map $\psi$ on the factor indexed by
  $\sigma\colon\alpha \to \alpha'$ is the composition of a natural
  projection from the product with the map
  \begin{displaymath}
    (1_{\diXhat_{\alpha'}})^{\diag K_{\sigma_{*}}}\colon
    (\diXhat_{\alpha'})^{\diag K_{\alpha'}}
    \longrightarrow
    (\diXhat_{\alpha'})^{\diag K_{\alpha}}    
  \end{displaymath}
  (where $\sigma_{*}\colon \diag K_{\alpha} \to \diag K_{\alpha'}$),
  and so $\cat M\bigl(W, \hom_{\diXhat}^{\cat C}(\diag K, \diag
  X)\bigr)$ is naturally isomorphic to the limit of the diagram
  \begin{displaymath}
    \xymatrix{
      {\prod_{\alpha\in\Ob(\cat C)}
        \cat M\bigl(W,(\diXhat_\alpha)^{\diag K_{\alpha}}\bigr)\quad}
      \ar[r]<.8ex>^-{\phi_{*}} \ar[r]<-.6ex>_-{\psi_{*}}
      &{\quad\prod_{(\sigma\colon \alpha \to \alpha') \in \cat C}
        \cat M\bigl(W,(\diXhat_{\alpha'})^{\diag K_{\alpha}}\bigr)\Period}
    }
  \end{displaymath}
  This is naturally isomorphic to the limit of the diagram
  \begin{displaymath}
    \xymatrix{
      {\prod_{\alpha\in\Ob(\cat C)}
        \SiS\bigl(\diag K_{\alpha},
        \map_{\diXhat}(W, \diXhat_{\alpha})\bigr) \quad}
      \ar[r]<.8ex>^-{\hat\phi_{*}} \ar[r]<-.6ex>_-{\hat\psi_{*}}
      &{\quad\prod_{(\sigma\colon \alpha \to \alpha') \in \cat C}
        \SiS\bigl(\diag K_{\alpha},
        \map_{\diXhat}(W, \diXhat_{\alpha'})\bigr)}
    }
  \end{displaymath}
  (see \propref{prop:adjointness}) which is the definition of
  $\SiS^{\cat C}\bigl(\diag K, \map_{\diXhat}(W,\diag X)\bigr)$.
\end{proof}

%--------------------------------------------------------------------
%--------------------------------------------------------------------
\section{Total objects}
\label{sec:TotOb}

We define the total object of a cosimplicial object in
\defref{def:Tot}, the total object of a multicosimplicial object in
\defref{def:nTot}, and show in \thmref{thm:IsoTot} that the total
object of a multicosimplicial object is isomorphic to the total
object of its diagonal cosimplicial object.

%--------------------------------------------------------------------
\subsection{The total object of a cosimplicial object}

\begin{defn}
  \label{def:Tot}
  If $\cat M$ is a model category, $\diag X$ is a cosimplicial object
  in $\cat M$, and $\diXhat$ is a Reedy simplicial frame on $\diag X$
  (see \defref{def:frame}), then the \emph{total object} $\Tot\diag X$
  of $\diag X$ is the object of $\cat M$ that is the end (see
  \cite{MCATL}*{Def.~18.3.2} or \cite{McL:categories}*{pages 218--223}
  or \cite{borceux-II}*{page 329}) of the functor
  $\diXhat^{\Delta}\colon \bD\times \bD\op \to \cat M$.  This is a
  subobject of the product
  \begin{displaymath}
    \prod_{k \ge 0} (\diXhat^{k})^{\Delta[k]}
  \end{displaymath}
  and is denoted $\hom_{\diXhat}^{\bD}(\Delta,\diag X)$ in
  \cite{MCATL}*{Def.~19.2.2} and $\int_{[k]} \diag
  (\diXhat^{k})^{\Delta[k]}$ in \cite{McL:categories}*{pages
    218--223}.
\end{defn}

\begin{prop}
  \label{prop:TotUnq}
  If $\cat M$ is a model category, $\diag X$ is a Reedy fibrant
  cosimplicial object in $\cat M$, and $\diXhat$ and $\diXhat'$ are
  two Reedy simplicial frames on $\diag X$, then there is an
  essentially unique zig-zag of weak equivalences connecting
  $\Tot\diag X$ defined using $\diXhat$ and $\Tot\diag X$ defined
  using $\diXhat'$.
\end{prop}

\begin{proof}
  This follows from \propref{prop:frameunique} and
  \cite{MCATL}*{Cor.~19.7.4}.
\end{proof}

\begin{ex}
  \label{ex:Tot}
  If $\cat M$ is the category of simplicial sets, $\diag X$ is a
  cosimplicial object in $\cat M$, and $\diXhat$ is the standard
  simplicial frame on $\diag X$, then $\Tot\diag X$ is the simplicial
  set of maps of cosimplicial simplicial sets from $\Delta$ to $\diag
  X$, i.e., a subset of the product simplicial set $\prod_{k \ge 0}
  (\diag X^{k})^{\Delta[k]}$.

  If $\cat M$ is the category of topological spaces, $\diag X$ is a
  cosimplicial object in $\cat M$, and $\diXhat$ is the standard
  simplicial frame on $\diag X$, then $\Tot\diag X$ is the topological
  space of maps of cosimplicial spaces from $\real{\Delta}$ to $\diag
  X$, i.e., a subset of the product space $\prod_{k \ge 0} (\diag
  X^{k})^{\real{\Delta[k]}}$.
\end{ex}

\begin{prop}
  \label{prop:WEtot}
  If $\cat M$ is a model category and $\diag X$ is a Reedy fibrant
  cosimplicial object in $\cat M$, and $\diXhat$ is a Reedy simplicial
  frame on $\diXhat$, then $\Tot\diag X$ is a fibrant object of $\cat
  M$.
\end{prop}

\begin{proof}
  See \cite{MCATL}*{Thm.~19.8.2}.
\end{proof}

%--------------------------------------------------------------------
\subsection{The total object of a multicosimplicial object}

\begin{defn}
  \label{def:nTot}
  If $n$ is a positive integer, $\cat M$ is a model category, $\diag
  X$ is an $n$-cosimplicial object in $\cat M$, and $\diXhat$ is a
  Reedy simplicial frame on $\diag X$ (see \defref{def:frame}), then
  the \emph{total object} $\Tot\diag X$ of $\diag X$ is the object of
  $\cat M$ that is the end (see \cite{MCATL}*{Def.~18.3.2},
  \cite{McL:categories}*{pages 218--223}, or \cite{borceux-II}*{page
    329}) of the functor $\diXhat^{\Delta^{(n)}}\colon \bD^{n}\times
  (\bD^{n})\op \to \cat M$.  This is a subobject of the product
  \begin{displaymath}
    \prod_{k_{1}\ge 0, k_{2}\ge 0,\ldots, k_{n}\ge 0}
    \bigl(\diXhat^{(k_{1},k_{2},\ldots,
      k_{n})}\bigr)^{(\Delta[k_{1}] \times \Delta[k_{2}] \times \cdots
      \times \Delta[k_{n}])}
  \end{displaymath}
  and is denoted $\hom_{\diXhat}^{\bD^{n}}(\Delta^{(n)},\diag X)$ in
  \cite{MCATL}*{Def.~19.2.2} and
  \begin{displaymath}
    \int_{([k_{1}], [k_{2}],\ldots,[k_{n}])}
    \bigl(\diXhat^{(k_{1},k_{2},\ldots,
    k_{n})}\bigr)^{(\Delta[k_{1}] \times \Delta[k_{2}] \times \cdots
    \times \Delta[k_{n}])}
  \end{displaymath}
  in \cite{McL:categories}*{pages 218--223}.
\end{defn}

\begin{prop}
  \label{prop:MTotUnq}
  If $\cat M$ is a model category, $\diag X$ is a Reedy fibrant
  multicosimplicial object in $\cat M$, and $\diXhat$ and $\diXhat'$
  are two Reedy simplicial frames on $\diag X$, then there is an
  essentially unique zig-zag of weak equivalences connecting
  $\Tot\diag X$ defined using $\diXhat$ and $\Tot\diag X$ defined
  using $\diXhat'$.
\end{prop}

\begin{proof}
  This follows from \propref{prop:frameunique} and
  \cite{MCATL}*{Cor.~19.7.4}.
\end{proof}

\begin{ex}
  \label{ex:nTot}
  If $n$ is a positive integer, $\cat M$ is the category of simplicial
  sets, $\diag X$ is an $n$-cosimplicial object in $\cat M$, and
  $\diXhat$ is the standard simplicial frame on $\diag X$, then
  $\Tot\diag X$ is the simplicial set of maps of $n$-cosimplicial
  simplicial sets from $\Delta^{(n)}$ to $\diag X$, i.e., a subset of
  the product simplicial set
  \begin{displaymath}
    \prod_{([k_{1}], [k_{2}], \ldots, [k_{n}])}
    (\diag X^{(k_{1},k_{2},\ldots, k_{n})})^{(\Delta[k_{1}]
      \times \Delta[k_{2}] \times \cdots \times \Delta[k_{n}])} \Period
  \end{displaymath}

  If $n$ is a positive integer, $\cat M$ is the category of
  topological spaces, $\diag X$ is an $n$-cosimplicial object in $\cat
  M$, and $\diXhat$ is the standard simplicial frame on $\diag X$,
  then $\Tot\diag X$ is the topological space of maps of
  $n$-cosimplicial spaces from $\real{\Delta^{(n)}}$ to $\diag X$,
  i.e., a subspace of the product space $\prod_{([k_{1}], [k_{2}],
    \ldots, [k_{n}])} (\diag X^{(k_{1},k_{2},\ldots,
    k_{n})})^{\real{(\Delta[k_{1}] \times \Delta[k_{2}] \times \cdots
      \times \Delta[k_{n}])}}$.
\end{ex}

\begin{prop}
  \label{prop:WEntot}
  If $n$ is a positive integer, $\cat M$ is a model category, $\diag
  X$ is a Reedy fibrant $n$-cosimplicial object in $\cat M$, and
  $\diXhat$ is a Reedy simplicial frame on $\diXhat$, then $\Tot\diag
  X$ is a fibrant object of $\cat M$.
\end{prop}

\begin{proof}
  Since $\Delta^{(n)}$ is Reedy cofibrant, this follows from
  \cite{MCATL}*{Cor.~19.7.3}.
\end{proof}

%--------------------------------------------------------------------
\subsection{The total object of the diagonal}
\label{sec:TotDiag}

In \thmref{thm:IsoTot} we use \thmref{thm:LeftKan} to show that the
total object of an $n$-cosimplicial object in an arbitrary model
category is isomorphic to the total object of its diagonal
cosimplicial object (see also \cite{dwy-mil-neis}*{Remark on p.~172}
and \cite{shipley}*{Prop.~8.1}).

\begin{prop}
  \label{prop:DiagRee}
  If $n$ is a positive integer, $\cat M$ is a model category, $\diag
  X$ is an $n$-cosimplicial object in $\cat M$, and $\diXhat$ is a
  Reedy simplicial frame on $\diag X$, then $\diagon\diXhat$ is a
  Reedy simplicial frame on $\diagon\diag X$.
\end{prop}

\begin{proof}
  Since a Reedy simplicial frame on a Reedy fibrant $n$-cosimplicial
  object is a Reedy fibrant diagram in $(\cat M^{\bD^{n}})^{\bD\op}
  \iso (\cat M^{\bD\op})^{\bD^{n}}$, and the different Reedy model
  category structures on that category coincide (see
  \cite{MCATL}*{Thm.~15.5.2}), this follows from \corref{cor:fibrant}.
\end{proof}

\begin{thm}
  \label{thm:IsoTot}
  If $n$ is a positive integer, $\cat M$ is a model category, $\diag
  X$ is an $n$-cosimplicial object in $\cat M$, and $\diXhat$ is a
  Reedy simplicial frame on $\diag X$, then $\diagon\diXhat$ is a
  Reedy simplicial frame on $\diagon\diag X$ (see
  \propref{prop:DiagRee}) (which, by abuse of notation, we will also
  denote by $\diXhat$) and there is a natural isomorphism
  \begin{displaymath}
    \Tot\diag X \iso \Tot(\diagon\diag X)
  \end{displaymath}
  from the total object of $\diag X$ to the total object of the
  diagonal cosimplicial object of $\diag X$.
\end{thm}

\begin{proof}
  For every object $W$ of $\cat M$ there
  are natural isomorphisms of sets
  \begin{align*}
    \cat M(W, \Tot\diag X)
    &= \cat M\bigl(W,
    \hom_{\diXhat}^{\bD^{n}}(\Delta^{(n)},\diag X)\bigr)\\
    &\iso \SiS^{\bD^{n}}\bigl(\Delta^{(n)},
    \map_{\diXhat}(W,\diag X)\bigr)
    &&\text{(see \thmref{thm:EndAdj})}\\
    &\iso \SiS^{\bD}\bigl(\Delta,
    \diagon\map_{\diXhat}(W,\diag X)\bigr)
    &&\text{(see \thmref{thm:LeftKan})}\\
    &\iso \SiS^{\bD}\bigl(\Delta,
    \map_{\diXhat}(W,\diagon\diag X)\bigr)\\
    &\iso \cat M\bigl(W,
    \hom_{\diXhat}^{\bD}(\Delta, \diagon\diag X)\bigr)
    &&\text{(see \thmref{thm:EndAdj})}\\
    &= \cat M\bigl(W, \Tot(\diagon\diag X)\bigr)
  \end{align*}
  and the Yoneda lemma implies that the composition of those is
  induced by a natural isomorphism $\Tot\diag X \iso \Tot(\diagon\diag
  X)$.
\end{proof}

%--------------------------------------------------------------------
\section{Homotopy limits and total objects}
\label{sec:holim}

We show that the Bousfield-Kan map from the total object of a
multicosimplicial object to its homotopy limit is a weak equivalence
for Reedy fibrant multicosimplicial objects.  We also show that this
behaves well with respect to passing to the diagonal cosimplicial
object of a multicosimplicial object.

%--------------------------------------------------------------------
\subsection{Cosimplicial objects}

\begin{defn}
  \label{def:BKmapSS}
  The \emph{Bousfield-Kan map} is the map of cosimplicial simplicial
  sets $\phi\colon \B\overcat{\bD}{-} \to \Delta$ that for $k \ge 0$
  and $n \ge 0$ takes the $n$-simplex
  \begin{displaymath}
    \Bigl(\bigl([i_{0}] \xrightarrow{\sigma_{0}} [i_{1}]
    \xrightarrow{\sigma_{1}} \cdots \xrightarrow{\sigma_{n-1}}
    [i_{n}]\bigr), \tau\colon [i_{n}] \to [k] \Bigr)
  \end{displaymath}
  of $\B\overcat{\bD}{[k]}$ to the $n$-simplex
  \begin{displaymath}
    [\tau\sigma_{n-1}\sigma_{n-2}\cdots\sigma_{0}(i_{0}),
    \tau\sigma_{n-1}\sigma_{n-2}\cdots\sigma_{1}(i_{1}),
    \ldots, \tau\sigma_{n-1}(i_{n-1}), \tau(i_{n})]
  \end{displaymath}
  of $\Delta[k]$ (see \cite{MCATL}*{Def.~18.7.1}).  It is a weak
  equivalence of Reedy cofibrant cosimplicial simplicial sets (see
  \cite{MCATL}*{Prop.~18.7.2}).
\end{defn}

\begin{thm}
  \label{thm:HolimTot}
  If $\cat M$ is a simplicial model category, $\diag X$ is a Reedy
  fibrant cosimplicial object in $\cat M$, and $\diXhat$ is a Reedy
  simplicial frame on $\diag X$, then the map
  \begin{displaymath}
    \Tot\diag X \iso
    \hom_{\diXhat}^{\bD}(\Delta,\diag X)
    \xrightarrow{\hom_{\diXhat}^{\bD}(\phi,1_{\diag X})}
    \hom_{\diXhat}^{\bD}\bigl(\B\overcat{\bD}{-},\diag X\bigr)
    \iso \holim\diag X
  \end{displaymath}
  (where $\phi\colon \B\overcat{\bD}{-} \to \Delta$ is the
  Bousfield-Kan map of cosimplicial simplicial sets; see
  \defref{def:BKmapSS}) is a natural weak equivalence of fibrant
  objects $\Tot\diag X \we \holim\diag X$, which we will also call the
  \emph{Bousfield-Kan map}.
\end{thm}

\begin{proof}
  Since the Bousfield-Kan map of cosimplicial simplicial sets is a
  weak equivalence of Reedy cofibrant cosimplicial sets, this follows
  from \cite{MCATL}*{Cor.~19.7.5}.
\end{proof}

%--------------------------------------------------------------------
\subsection{Multicosimplicial objects}

\begin{lem}
  \label{lem:OvCtProd}
  Let $n$ be a positive integer, let $\cat C_{i}$ be a small category
  for $1 \le i \le n$, and let $\cat C = \prod_{1 \le i \le n} \cat
  C_{i}$.  If $\alpha = (\alpha_{1}, \alpha_{2}, \ldots, \alpha_{n})$
  is an object of $\cat C$, then the overcategory $\overcat{\cat
    C}{\alpha} \iso \prod_{1 \le i \le n} \overcat{\cat
    C_{i}}{\alpha_{i}}$, and its classifying space (or nerve) is
  $\B\overcat{\cat C}{\alpha} \iso \prod_{1 \le i \le n}
  \B\overcat{\cat C_{i}}{\alpha_{i}}$.
\end{lem}

\begin{proof}
  This follows directly from the definitions.
\end{proof}

\begin{defn}
  \label{def:multBKmapSS}
  If $n$ is a positive integer, then the \emph{product Bousfield-Kan
    map} of $n$-cosimplicial simplicial sets $\phi^{n}\colon
  \B\overcat{\bD^{n}}{-} \to \Delta^{(n)}$ is the composition
  \begin{displaymath}
    \B\overcat{\bD^{n}}{-} \iso
    \prod_{1 \le i \le n} \B\overcat{\bD}{-}
    \xrightarrow{\quad\phi^{n}\quad}
    \Delta^{(n)}
  \end{displaymath}
  (see \lemref{lem:OvCtProd}) where $\phi$ is the Bousfield-Kan map of
  cosimplicial simplicial sets (see \defref{def:BKmapSS}).
\end{defn}

\begin{thm}
  \label{thm:multiHolimTot}
  If $n$ is a positive integer, $\cat M$ is a simplicial model
  category, $\diag X$ is a Reedy fibrant $n$-cosimplicial object in
  $\cat M$, and $\diXhat$ is a Reedy simplicial frame on $\diag X$,
  then the map
  \begin{displaymath}
    \Tot\diag X \iso
    \hom_{\diXhat}^{\bD^{n}}(\Delta^{(n)},\diag X)
    \xrightarrow{\hom_{\diXhat}^{\bD^{n}}(\phi^{n},1_{\diag X})}
    \hom_{\diXhat}^{\bD^{n}}\bigl(\B\overcat{\bD^{n}}{-},\diag X\bigr)
    \iso \holim\diag X
  \end{displaymath}
  (where $\phi^{n}\colon \B\overcat{\bD^{n}}{-} \to \Delta^{(n)}$ is
  the product Bousfield-Kan map of $n$-cosimplicial simplicial sets;
  see \defref{def:multBKmapSS}) is a natural weak equivalence of
  fibrant objects $\Tot\diag X \we \holim\diag X$, which we will also
  call the \emph{product Bousfield-Kan map}.
\end{thm}

\begin{proof}
  Since the product Bousfield-Kan map of $n$-cosimplicial simplicial
  sets is a weak equivalence of Reedy cofibrant $n$-cosimplicial sets
  (see \lemref{lem:prodsimpcof}), this follows from
  \cite{MCATL}*{Cor.~19.7.5}.
\end{proof}

%--------------------------------------------------------------------
\subsection{The homotopy limit and total object of the diagonal}

We first show that for an objectwise fibrant multicosimplicial object
the canonical map from the homotopy limit to the homotopy limit of the
diagonal cosimplicial object is a weak equivalence, and then we show
that all the maps we've defined between the homotopy limits and total
objects of a multicosimplicial object and its diagonal cosimplicial
object commute.

\begin{prop}
  \label{prop:DiagHocof}
  If $n$ is a positive integer, then the diagonal embedding $D\colon
  \bD \to \bD^{n}$ (see \defref{def:diag}) is homotopy left cofinal
  (see \cite{MCATL}*{Def.~19.6.1}).
\end{prop}

\begin{proof}
  For an object $([p_{1}], [p_{2}], \ldots, [p_{n}])$ of $\bD^{n}$, an
  object of $\bovercat{\bD}{([p_{1}], [p_{2}], \ldots, [p_{n}])}$ is a
  map
  \begin{displaymath}
    (\alpha_{1}, \alpha_{2}, \ldots, \alpha_{n}) \colon
    \bigl([k], [k], \ldots, [k]\bigr) \longrightarrow
    \bigl([p_{1}], [p_{2}], \ldots, [p_{n}]\bigr)
  \end{displaymath}
  in $\bD^{n}$, where each $\alpha_{i}\colon [k] \to [p_{i}]$ is a map
  in $\bD$, i.e., a $k$-simplex of $\Delta[p_{i}]$.  Thus,
  $(\alpha_{1}, \alpha_{2}, \ldots, \alpha_{n})$ is a $k$-simplex of
  $\Delta[p_{1}]\times \Delta[p_{2}]\times \cdots\times
  \Delta[p_{n}]$, and the category $\bovercat{\bD}{([p_{1}], [p_{2}],
    \ldots, [p_{n}])}$ is the category of simplices of
  $\Delta[p_{1}]\times \Delta[p_{2}]\times \cdots\times \Delta[p_{n}]$
  (see \lemref{lem:ColimCtSmp}).  Since the nerve of the category of
  simplices of a simplicial set is weakly equivalent to that
  simplicial set (see \cite{MCATL}*{Thm.~18.9.3}),
  $\B\bovercat{\bD}{([p_{1}], [p_{2}], \ldots, [p_{n}])}$ is weakly
  equivalent to $\Delta[p_{1}]\times \Delta[p_{2}]\times \cdots\times
  \Delta[p_{n}]$, and is thus contractible.
\end{proof}

\begin{thm}
  \label{thm:WEdiag}
  If $n$ is a positive integer, $\cat M$ is a model category, and
  $\diag X$ is an objectwise fibrant $n$-cosimplicial object in $\cat
  M$, then the natural map $\holim_{\bD^{n}}\diag X \to
  \holim_{\bD}\diagon\diag X$ induced by the diagonal embedding
  $D\colon \bD \to \bD^{n}$ is a weak equivalence.
\end{thm}

\begin{proof}
  This follows from \propref{prop:DiagHocof} and
  \cite{MCATL}*{Thm.~19.6.7}.
\end{proof}

\begin{thm}
  \label{thm:holimdiag}
  If $n$ is a positive integer, $\cat M$ is a model category, $\diag
  X$ is an $n$-cosimplicial object in $\cat M$, and $\diXhat$ is a
  Reedy simplicial frame on $\diag X$, then the diagram
  \begin{displaymath}
    \xymatrix{
      {\Tot\diag X} \ar[r] \ar[d]
      & {\holim_{\bD^{n}}\diag X} \ar[d]\\
      {\Tot\diagon\diag X} \ar[r]
      & {\holim_{\bD}\diagon\diag X}
    }
  \end{displaymath}
  (where the upper horizontal map is the product Bousfield-Kan map
  (see \thmref{thm:multiHolimTot}), the lower horizontal map is the
  Bousfield-Kan map (see \thmref{thm:HolimTot}), the left vertical map
  is the isomorphism of \thmref{thm:IsoTot}, and the right vertical
  map is the natural map induced by the diagonal embedding $D\colon
  \bD \to \bD^{n}$ (see \cite{MCATL}*{Prop.~19.1.8})) commutes.  If
  $\diag X$ is objectwise fibrant, then the vertical maps in that
  diagram are weak equivalences.  If $\diag X$ is Reedy fibrant, then
  all of the maps in that diagram are weak equivalences.
\end{thm}

\begin{proof}
  It is sufficient to show that if $W$ is an object of $\cat M$, then
  the diagram
  \begin{displaymath}
    \xymatrix{
      {\cat M(W, \Tot\diag X)} \ar[r] \ar[d]
      & {\cat M(W, \holim_{\bD^{n}}\diag X)} \ar[d]\\
      {\cat M(W, \Tot\diagon\diag X)} \ar[r]
      & {\cat M(W, \holim_{\bD}\diagon\diag X)}
    }
  \end{displaymath}
  commutes.  \thmref{thm:EndAdj} gives us natural isomorphisms
  \begin{align*}
    \cat M(W, \Tot\diag X)
    &= \cat M\bigl(W, \hom_{\diXhat}^{\bD^{n}}(\Delta^{(n)}, \diag
    X)\bigr)\\
    &\iso \SiS^{\bD^{n}}\bigl(\Delta^{(n)},
    \map_{\diXhat}(W,\diag X)\bigr)\\[1ex]
    \cat M(W, \holim\diag X)
    &= \cat M\bigl(W, \hom_{\diXhat}^{\bD^{n}}(\B\overcat{\bD^{n}}{-},
    \diag X)\bigr)\\
    &\iso \SiS^{\bD^{n}}\bigl(\B\overcat{\bD^{n}}{-},
    \map_{\diXhat}(W,\diag X)\bigr)\\[1ex]
    \cat M(W, \Tot\diagon\diag X)
    &= \cat M\bigl(W, \hom_{\diXhat}^{\bD}(\Delta, \diagon\diag
    X)\bigr)\\
    &\iso \SiS^{\bD}\bigl(\Delta,
    \map_{\diXhat}(W,\diagon\diag X)\bigr)\\[1ex]
    \cat M(W, \holim\diagon\diag X)
    &= \cat M\bigl(W, \hom_{\diXhat}^{\bD}(\B\overcat{\bD}{-},
    \diagon\diag X)\bigr)\\
    &\iso \SiS^{\bD}\bigl(\B\overcat{\bD}{-},
    \map_{\diXhat}(W, \diagon\diag X)\bigr)
  \end{align*}
  and so this is equivalent to showing that the diagram
  \begin{displaymath}
    \xymatrix{
      {\SiS^{\bD^{n}}\bigl(\Delta^{(n)},
        \map_{\diXhat}(W,\diag X)\bigr)}
      \ar[r] \ar[d]
      & {\SiS^{\bD^{n}}\bigl(\B\overcat{\bD^{n}}{-},
        \map_{\diXhat}(W,\diag X)\bigr)}
      \ar[d]\\
      {\SiS^{\bD}\bigl(\Delta, \map_{\diXhat}(W,\diagon\diag X)\bigr)}
      \ar[r]
      & {\SiS^{\bD}\bigl(\B\overcat{\bD}{-},
        \map_{\diXhat}(W,\diagon\diag X)\bigr)}
    }
  \end{displaymath}
  commutes.  If $f \in \SiS^{\bD^{n}}\bigl(\Delta^{(n)},
  \map_{\diXhat}(W, \diag X)\bigr)$, then the image of $f$ under the
  counterclockwise composition is the composition of
  \begin{displaymath}
    \xymatrix@C=3em{
      {\B\overcat{\bD}{-}} \ar[r]^-{\phi}
      & {\Delta} \ar[r]^-{\alpha}
      & {\diagon\bigl(\Delta^{(n)}\bigr)} \ar[r]^-{\diagon f}
      & {\diagon\diag X}
    }
  \end{displaymath}
  (where $\phi$ is the Bousfield-Kan map of \defref{def:BKmapSS} and
  $\alpha$ is as in \thmref{thm:LeftKan}) and the image
  of $f$ under the clockwise composition is
  \begin{displaymath}
    \xymatrix@C=3em{
      {\B\overcat{\bD}{-}} \ar[r]^-{D_{*}}
      & {\diagon\bigl(\B\overcat{\bD^{n}}{-}\bigr)} \ar[r]^-{\phi^{n}}
      & {\diagon(\Delta^{(n)})} \ar[r]^-{\diagon f}
      & {\diagon\diag X}
    }
  \end{displaymath}
  (where $D_{*}$ is the map induced by the diagonal embedding $D\colon
  \bD \to \bD^{n}$ and $\phi^{n}$ is the product Bousfield-Kan map of
  \defref{def:multBKmapSS}).  Since the diagram
  \begin{displaymath}
    \xymatrix@C=3em{
      {\B\overcat{\bD}{-}} \ar[r]^-{D_{*}} \ar[d]_{\phi}
      & {\diagon\B\overcat{\bD^{n}}{-}} \ar[d]^{\diagon \phi^{n}}\\
      {\Delta} \ar[r]_-{\alpha}
      & {\diagon \Delta^{(n)}}
    }
  \end{displaymath}
  commutes, these two compositions are equal, and so our diagram
  commutes.
\end{proof}

%--------------------------------------------------------------------
%--------------------------------------------------------------------
%--------------------------------------------------------------------
%--------------------------------------------------------------------

%--------------------------------------------------------------------


\begin{bibdiv} 
  \begin{biblist}

\bib{B-E-J-M}{article}{
   author={Bauer, K.},
   author={Eldred, R.},
   author={Johnson, B.},
   author={McCarthy, R.},
   title={Combinatorial models for Taylor polynomials of functors},
   date={2015},
   eprint={http://arxiv.org/abs/1506.02112}
}

\bib{borceux-I}{book}{
   author={Borceux, Francis},
   title={Handbook of categorical algebra. 1},
   series={Encyclopedia of Mathematics and its Applications},
   volume={50},
   note={Basic category theory},
   publisher={Cambridge University Press, Cambridge},
   date={1994},
   pages={xvi+345},
}

\bib{borceux-II}{book}{
   author={Borceux, Francis},
   title={Handbook of categorical algebra. 2},
   series={Encyclopedia of Mathematics and its Applications},
   volume={51},
   note={Categories and structures},
   publisher={Cambridge University Press, Cambridge},
   date={1994},
   pages={xviii+443},
}

\bib{YM}{book}{
   author={Bousfield, A. K.},
   author={Kan, D. M.},
   title={Homotopy limits, completions and localizations},
   series={Lecture Notes in Mathematics, Vol. 304},
   publisher={Springer-Verlag, Berlin-New York},
   date={1972},
   pages={v+348},
}

\bib{dwy-mil-neis}{article}{
   author={Dwyer, William},
   author={Miller, Haynes},
   author={Neisendorfer, Joseph},
   title={Fibrewise completion and unstable Adams spectral sequences},
   journal={Israel J. Math.},
   volume={66},
   date={1989},
   number={1-3},
   pages={160--178},
}

\bib{MCATL}{book}{
   author={Hirschhorn, Philip S.},
   title={Model categories and their localizations},
   series={Mathematical Surveys and Monographs},
   volume={99},
   publisher={American Mathematical Society, Providence, RI},
   date={2003},
   pages={xvi+457},
}

\bib{McL:categories}{book}{
   author={MacLane, Saunders},
   title={Categories for the working mathematician},
   note={Graduate Texts in Mathematics, Vol. 5},
   publisher={Springer-Verlag, New York-Berlin},
   date={1971},
   pages={ix+262},
}

\bib{schubert}{book}{
   author={Schubert, Horst},
   title={Categories},
   note={Translated from the German by Eva Gray},
   publisher={Springer-Verlag, New York-Heidelberg},
   date={1972},
   pages={xi+385},
}

\bib{shipley}{article}{
   author={Shipley, Brooke E.},
   title={Convergence of the homology spectral sequence of a cosimplicial
   space},
   journal={Amer. J. Math.},
   volume={118},
   date={1996},
   number={1},
   pages={179--207},
}



  \end{biblist}
\end{bibdiv}
\end{document}